\documentclass[a4paper,reqno]{amsart}
\usepackage{biblatex}
\renewbibmacro{in:}{%
  \ifentrytype{article}{}{\printtext{\bibstring{in}\intitlepunct}}}
\usepackage{amsfonts}
\usepackage{amsthm}
\usepackage{amsmath}
\usepackage{datetime}
\usepackage{comment}
\usepackage{soul}
\usepackage{enumitem}
\usepackage[dvipsnames]{xcolor}
\usepackage{tikz}
\usetikzlibrary{plotmarks}
\usetikzlibrary{decorations.pathreplacing}
\numberwithin{equation}{section}

\newtheorem{theorem}{Theorem}[section]
\newtheorem{corollary}[theorem]{Corollary}
\newtheorem{lemma}[theorem]{Lemma}
\newtheorem{remark}[theorem]{Remark}
\newtheorem{definition}[theorem]{Definition}

\numberwithin{figure}{section}

\title[Discrete $q$-Freud $\mathrm{II}$ orthogonal polynomials]{Asymptotics of Discrete $q$-Freud $\mathrm{II}$ orthogonal polynomials from the $q$-Riemann Hilbert Problem}

\author{Nalini Joshi$^{1}$}
\email{nalini.joshi@sydney.edu.au}
\address{$^{1}$School of Mathematics and Statistics F07, University of Sydney, Sydney NSW 2006, Australia, ORCID ID: 0000-0001-7504-4444}

\author{Tomas Lasic Latimer$^{2}$}
\email{tlas5434@uni.sydney.edu.au}
\address{$^{2}$School of Mathematics and Statistics F07, University of Sydney, Sydney NSW 2006, Australia, ORCID ID: 0000-0001-6859-7788}


\date{}
\begin{document}
\begin{abstract}
We investigate a Riemann-Hilbert problem (RHP), whose solution corresponds to a group of $q$-orthogonal polynomials studied earlier by Ismail \textit{et al}. Using RHP theory we determine new asymptotic results in the limit as the degree of the polynomials approach infinity. The RHP formulation also enables us to obtain further properties. In particular, we consider how the class of polynomials and their asymptotic behaviours change under translations of the $q$-discrete lattice and determine the asymptotics of related $q$-Painlev\'e equations.
\end{abstract}


\maketitle
\tableofcontents

\section{Introduction}
Orthogonal polynomials are a key component of a wide array of mathematical problems. They provide the basis of solutions of Sturm-Liouville problems \cite{GOMEZULLATE2009352}, their zeros are related to the eigenvalue distribution of random matrices \cite{assche2018matrices}, they describe transition probabilities in birth-death models \cite{sasaki2009exactly} and are used in numerical spectral approximation methods, just to name a few examples. Their importance in describing physical phenomena was recognised over two hundred years ago (Legendre, Laplace) and they continue to be pivotal in describing mathematical and physical problems. 

In this paper we study a class of $q$-orthogonal polynomials and deduce new results concerning their asymptotic behaviour as the degree tends to infinity. The orthogonality measure of these polynomials is supported on the discrete lattice, $q^k$, for $k\in\mathbb{Z}$, where $0<q<1$. We also determine some properties of $q$-Painlev\'e equations associated with these $q$-orthogonal polynomials.

\subsection{Notation}\label{notation section}
For completeness, we recall some well known definitions and notations from the calculus of $q$-differences. These definitions can be found in \cite{Ernst2012}. Throughout the paper we will assume $q \in \mathbb{R}$ and $0<q<1$.
\begin{definition}\label{q stuff defined}
We define the Pochhammer symbol, $q$-derivative and Jackson integral as follows:
\begin{enumerate}
\item The Pochhammer symbol $(x;q)_{\infty}$ is
\begin{equation*}
    (x;q)_{\infty} = \prod_{j=0}^{\infty}(1-xq^{j}) \,.
\end{equation*}
We denote product of Pochhammer symbols in multiple variables by $(x_1,x_2;q)_{\infty}$,
\begin{equation*}
    (x_1,x_2;q)_{\infty} = \prod_{j=0}^{\infty}(1-x_1q^{j})(1-x_2q^{j}) \,.
\end{equation*}
\item The $q$-derivative is defined by
\begin{equation}\label{q_derivative}
    D_qf(x) = \frac{f(qx)-f(x)}{x(q-1)}.
\end{equation}
Note that 
\[ D_q x^n = [n]_qx^{n-1} ,\]
where 
\[ [n]_q = \frac{q^n-1}{q-1} .\]

\item The Jackson integral from $q^{j}$ to $q^{i}$ for some integers $i<j$ is given by
\[ \int_{q^j}^{q^i}f(x)d_qx = \sum_{k=i}^{j}f(q^k)q^k .\]
The Jackson integral from $-q^{j}$ to $q^{i}$ for some integers $i,j$ is given by
\[ \int_{-q^j}^{q^i}f(x)d_qx = \sum_{k=j}^{\infty}f(-q^k)q^k + \sum_{k=i}^{\infty}f(q^k)q^k.\]

\end{enumerate}
\end{definition}

\begin{definition}
In this paper $\mathbb{N}$ will denote the set of natural numbers including zero (i.e. 0,1,2,3,..), unless otherwise stated.
\end{definition}

We recall the definition of an {\em appropriate} Jordan curve and {\em admissible} weight function given in \cite[Definition 1.2]{qRHP} (with slight modification).
\begin{definition}\label{admissable}
A positively oriented Jordan curve $\Gamma$ in $\mathbb C$ with interior $\mathcal D_-\subset\mathbb C$ and exterior $\mathcal D_+\subset\mathbb C$ is called {\em appropriate} if 
\[
\pm q^k\in \begin{cases}
&\mathcal D_- \quad {\mathrm if}\, k\ge 0,\; (k\in \mathbb{Z}),\\
&\mathcal D_+ \quad {\mathrm if}\, k< 0,\; (k\in \mathbb{Z}),
\end{cases}
\]
and,
\[ e^{\frac{i(\pi + 2n\pi)}{4}}q^{-k} \in \mathcal{D}_+,\; (n\in 0,1,2,3\; \mathrm{and}\, k\in \mathbb{N}).\]
\end{definition}

\begin{definition}[\cite{qRHP}]\label{formal h def}
Define $h_q:\mathbb{C}\setminus (\{0\}\cup \{\pm q^k\}_{k=-\infty}^\infty) \to \mathbb{C}$ by
\begin{equation}\label{h new def}
h_q(z) = \sum_{k=-\infty}^{\infty} \frac{2zq^{k}}{z^{2}-q^{2k}}  = \sum_{k=-\infty}^{\infty} \left( \frac{q^{k}}{z-q^{k}} + \frac{q^{k}}{z+q^{k}} \right),
\end{equation}
Note that $h_q(z)$ satisfies the $q$-difference equation 
\begin{equation}\label{h alpha diff}
    h_q(qz) = h_q(z).
\end{equation}
\end{definition}
In Appendix \ref{Properties of hq}, we show that $h_q(z)$ has certain unique properties.

\subsection{Background}
Let $\{P_n(x)\}_{n=0}^\infty$ be a class of monic polynomials which satisfy the orthogonality relation
\begin{equation}\label{original orthogonal}
    \int_{\mathbb{R}} P_n(x)P_m(x)w(x)dx = \gamma_n\delta_{n,m},
\end{equation}
for some weight function $w(x)$. Equation \eqref{original orthogonal} gives rise to the three term recurrence relation 
\begin{equation}\label{general recurrence coefficients}
    xP_n(x) = P_{n+1}(x) + \beta_nP_n(x) + \alpha_nP_{n-1}(x),
\end{equation}
where the recurrence coefficients are given by
\begin{equation}\label{intro recurrence eqn}
    \alpha_n = \frac{\gamma_n}{\gamma_{n-1}}\,,\,\beta_n = \frac{\int_\mathbb{R} xP_n(x)^2dx}{\gamma_n}.
\end{equation}
For even weight functions, $w(-x) = w(x)$, we find $\beta_n = 0$ for all $n$. 

Following the pioneering work of Freud and others, questions arising about the asymptotic locations of the zeros of $P_n(x)$ and behaviour of $\alpha_n$ as $n\to\infty$ have led to many developments in orthogonal polynomials and approximation theory. Motivated by the work of Deift \textit{et al.} \cite{Deift1999strong} we use the setting of the Riemann Hilbert Problem (RHP) to answer such questions (in Theorems \ref{main result 1}, \ref{main result 2} and \ref{main result 3}) for a class of $q$-orthogonal polynomials. In particular, we study polynomials which satisfy the orthogonality condition
\[ \int_{-\infty}^\infty P_n(x)P_m(x)(-x^4;q^4)_\infty d_qx = \gamma_n\delta_{n,m} .\]
Throughout this paper we will label these polynomials as $q$-Freud II polynomials to be consistent with the nomenclature of the DLMF \cite[Chapter 18]{NIST:DLMF}. $q$-Freud II polynomials were studied earlier by Ismail \textit{et al.}  \cite{ISMAIL2010518}, however Ismail \textit{et al.} considered orthogonality on a continuous measure over $\mathbb{R}$. We will show in Section \ref{Non-unique measure section} that this continuous class of polynomials can readily be extended to those with a discrete measure, which will be the focus of this paper. 

A significant development in the theory of orthogonal polynomials is the observation that the recurrence coefficients $\alpha_n$ often give rise to discrete Painlev\'e equations \cite{fokas1991discrete}. For example, the recurrence coefficients of orthogonal polynomials with weight function $e^{-x^4/4}dx$ satisfy the discrete equation \cite{freud1976coefficients}
\[ \alpha_n(\alpha_{n+1}+\alpha_{n}+\alpha_{n-1}) = n, \]
which is a case of the first discrete Painlev\'e equation, or $\mathrm{dP_{I}}$ \cite{joshi2019discrete}. 

The recurrence coefficients of $q$-Freud II polynomials also satisfy a discrete Painlev\'e equation, where the non-autonomous term in the equation is now iterated on a multiplicative lattice. (For the terminology distinguishing types of discrete Painlev\'e equations, we refer to Sakai \cite{s:01}.) As detailed by Ismail \textit{et al.} \cite{ISMAIL2010518} the recurrence coefficients of $q$-Freud II polynomials satisfy
\begin{equation} \label{qPainleve 4}
    \alpha_{n}(\alpha_{n+1}+q^{n-1}\alpha_{n}+q^{-2}\alpha_{n-1}-q^{2n-3}\alpha_{n+1}\alpha_{n}\alpha_{n-1}) = (q^{-n}-1)q^{1-n}.
\end{equation}
In Sections \ref{Non-unique measure section} and \ref{recurrence coefficients section} we extend this result initially determined by Ismail \textit{et al.} to a larger class of $q$-Freud II polynomials, using techniques similar to those found in \cite{Boelen}. 

Through the connection to RHP theory developed in this paper we obtain new insights into the solutions of Equation \eqref{qPainleve 4}. For example, Theorem \ref{main result 3} shows that there exists more than one real positive solution of Equation \eqref{qPainleve 4}, and we notice that the asymptotic behaviour as $n\to \infty$ varies between solutions. 

There is an important feature of $q$-difference equations that affects our discussion of $q$-Freud II polynomials. The associated weight function satisfies a $q$-difference equation, which gives rise to a family of weights involving a free $q$-periodic function $C(x)$, where $C(qx)=C(x)$. If a weight in this family with $C\not\equiv 1$ were to be chosen, the resulting family of orthogonal polynomials may have properties that differ from the class we consider. We expand on this point below.

\subsection{Defining $q$-Freud II polynomials}\label{Non-unique measure section}
We define the family of weight functions $u_m: \mathbb{C}\setminus \{e^{\frac{i\pi(1+2n)}{2m}}q^{-k}\}_{k=0}^\infty \to \mathbb{C}$, where $n=0,1,...,2m-1$, as
\begin{equation}\nonumber
    u_m(x) = \frac{1}{(-x^{2m},q^{2m})_\infty}, 
\end{equation} 
where $m \in \mathbb{N}$ \cite{ISMAIL2010518}. They satisfy the $q$-difference equation
\begin{equation}\label{general qdiff}
    D_q u_m(x) = \frac{-x^{2m-1}}{1-q}u_m(x) .
\end{equation} 
This is analogous to classical Freudian weights $v_m(x) = e^{-x^{2m}}$, which satisfy the differential equation
\begin{equation}\label{general diff}
    \frac{d}{dx} v_m(x) = -2mx^{2m-1}v_m(x) .
\end{equation} 
However, a key difference between these two relations is that Equation \eqref{general qdiff} is a discrete relation. In particular if $u_m(x)$ satisfies Equation \eqref{general qdiff} then $u_m(x)C(x)$ also does, for any function $C(x)$ satisfying $C(qx) = C(x)$. Consider the case
\begin{equation}\label{hermite weight}
    u_2(x) = \frac{1}{(-x^2,q^2)_\infty}. 
\end{equation} 
This weight gives rise to discrete $q$-Hermite $\mathrm{II}$ polynomials \cite[Chapter 18.27]{NIST:DLMF}. The sequence of discrete $q$-Hermite $\mathrm{II}$ polynomials, $\{H_n(x)\}_{n=0}^\infty$, are orthogonal with respect to any measure $u_2(x)C(x)$. In particular they satisfy the continuous orthogonality condition
\[ \int_{-\infty}^{\infty} H_n(x)H_m(x)u_2(x)dx = \gamma_n \delta_{n,m}, \]
on the real line, and also satisfy the discrete orthogonality condition
\[ \int_{-\infty}^{\infty} H_n(cx)H_m(cx)u_2(cx)d_qx = \gamma^{(c)}_{n} \delta_{n,m}, \]
for any constant $c$. In contrast, as we will show in Section \ref{recurrence coefficients section} this is not true for the weight,
\begin{equation}\label{w def}
    w(x) = u_4(x) = \frac{1}{(-x^4,q^4)_\infty},
\end{equation} 
which is the focus of this paper. Thus, when describing $q$-Freud $\mathrm{II}$ orthogonal polynomials, one also has to specify their orthogonality weight. 

For the remainder of this paper, we focus on the sequence of polynomials $\{P_n(x)\}$, $0\le n\in\mathbb N$, that satisfy
\[ \int_{-\infty}^{\infty} P_n(x)P_m(x)w(x)d_qx = \gamma_{n} \delta_{n,m}. \]
We will call these polynomials $qF_{II}$ polynomials. In Section \ref{recurrence coefficients section}, we discuss the implications of our results to polynomials orthogonal with respect to the weights of the form 
\begin{equation}\label{intro general c ortho}
    \int_{-\infty}^{\infty} P_n(cx)P_m(cx)w(cx)d_qx = \gamma^{(c)}_{n} \delta_{n,m}, 
\end{equation} 
for any constant $q<c\leq 1$. We will call these polynomials $qF_{II}^{(c)}$ polynomials.
\subsection{Main results}
We are now in a position to state the main results of this paper, which are listed as Theorems \ref{main result 1}, \ref{main result 2} and \ref{main result 3} below. The first main result concerns the asymptotic behaviour of orthogonal polynomials as their degree approaches infinity.
\begin{theorem}\label{main result 1}
Suppose that $\{P_n(z)\}_{n=0}^\infty$ is a family of monic polynomials, orthogonal with respect to the weight $w(z)d_qz$. Define $t = zq^{n/2}$. Then, as $n\to \infty$, for even $n\in \mathbb{N}$:
\[
    P_n(z) = \left\{\begin{array}{lr}
     (-1)^{n/2} q^{\frac{-n}{2}(\frac{n}{2}-1)} a(z)\left(\frac{\mu_2}{\eta_2}+ O(q^{n/4}) \right) & \text{for } |z|\leq q^{-n/4}, \\
       z^{n}a_{\infty}(t)\left(1 + O(q^{n/4})\right) & \text{for } |z|> q^{-n/4},
        \end{array}\right.
\]
where $a(z)$ is a solution of Equation \eqref{near 0 q diff}, defined in Lemma \ref{L01}, and, $a_{\infty}(t)$ is a solution of Equation \eqref{near infty q diff}, defined in Lemma \ref{Linfty1} (they are both independent of $n$). Similarly $P_{n-1}(z)$ has the asymptotic behaviour
\[
    P_{n-1}(z) = \left\{\begin{array}{lr}
     (-1)^{n/2} q^{\frac{n}{2}(\frac{n}{2}-1)}\gamma_{n-1}b(z)\left(\frac{\lambda_3}{c_{\Psi}}+ O(q^{n/4}) \right) & \text{for } |z|\leq q^{-n/4}, \\
       \frac{q^{n(\frac{n}{2}-1)}\gamma_{n-1}}{\mu_4\mathcal{H}c_\varPsi}z^{n}b_{\infty}(t)\left(1+ O(q^{n/4}) \right) & \text{for } |z|> q^{-n/4},
        \end{array}\right.
\]
where $b(z)$ is also a solution of Equation \eqref{near 0 q diff}, defined in Lemma \ref{L01}, and, $b_{\infty}(t)$ is a solution of Equation \eqref{near infty q diff}, defined in Lemma \ref{Linfty1}. 
\end{theorem}
Note that in the statement of Theorem \ref{main result 1}, $\mu_2$, $\eta_2$, $\mu_4$, $\mathcal{H}$ and $c_\varPsi$ are constants determined in Sections \ref{near-field rhp section} and \ref{far-field RHP section} which do not depend on $z$ or $n$. Our second main result concerns the asymptotic behaviour of recurrence coefficients and $L_2$ norm of $P_n$ as $n$ approaches infinity.
\begin{theorem}\label{main result 2}
Under the same hypotheses as Theorem \ref{main result 1}, we have for even $n\in\mathbb{N}$:
\begin{eqnarray*}
   \gamma_{n} &=& q^{\frac{n}{2}(1-n)}\left( A + O(q^{n/2})\right),\\
  \gamma_{n-1} &=& q^{\frac{n}{2}(3-n)}\left( B^{-1} + O(q^{n/2})\right),\\
   \alpha_n &=& q^{-n}(q+O(q^{n/2})),
\end{eqnarray*}
for some constants $A$ and $B$, where $\gamma_n$ and $\alpha_n$ are defined by Equations \eqref{original orthogonal} and \eqref{intro recurrence eqn} respectively.
\end{theorem}

Our third main theorem answers the question of uniqueness posed by Ismail \textit{et al.} in \cite[Remark 6.4]{ISMAIL2010518}.
\begin{theorem}\label{main result 3}
There exists infinitely many real positive solutions of the discrete equation
\begin{equation} 
    \alpha_{n}(\alpha_{n+1}+q^{n-1}\alpha_{n}+q^{-2}\alpha_{n-1}-q^{2n-3}\alpha_{n+1}\alpha_{n}\alpha_{n-1}) = (q^{-n}-1)q^{1-n}.
\end{equation}
Furthermore, in general 
\begin{equation}
    \lim_{n\to\infty} \frac{\alpha^{(c)}_n - \alpha^{(1)}_n}{q^{1-n}- \alpha^{(1)}_n} \neq 0,
\end{equation}
where $\{\alpha^{(c)}_n\}_{n=0}^\infty$ is the sequence of recurrence coefficients corresponding to polynomials with orthogonality condition given by Equation \eqref{intro general c ortho}.

\end{theorem}
Theorem \ref{main result 3} immediately follows from Theorem \ref{theorem diff asym} proved in Section \ref{recurrence coefficients section}.
\subsection{Outline}
This paper is structured as follows. In Section \ref{statement of rhp section} we state and solve a RHP (Definition \ref{quartic RHP}) whose solution is given in terms of $qF_{II}$ polynomials. We then make a series of transformations to this RHP in Section \ref{trans section}. By taking the limit $n\rightarrow \infty$ this motivates the form of a near-field and far-field RHP, whose solutions are determined in Sections \ref{near-field rhp section} and \ref{far-field RHP section} respectively. In Section \ref{glue section} we glue together the near and far-field solutions to approximate our initial RHP in the limit $n\to\infty$. Consequently we prove Theorems \ref{main result 1} and \ref{main result 2} in Section \ref{main proof section}. In Section \ref{recurrence coefficients section} we discuss the implications of our results to the recurrence coefficients of $qF_{II}$ polynomials and prove Theorem \ref{main result 3}. In Appendix \ref{Properties of hq} we prove some important properties of the function $h_q(z)$ which are used in solving the near and far-field RHPs. Finally, for completeness we prove a well-known result concerning RHPs whose jump approaches the identity in Appendix \ref{R to I section}.

\section{Statement of RHP}\label{statement of rhp section}
We begin the main arguments of this paper by introducing and solving a RHP (Definition \ref{quartic RHP}) whose solution is given in terms of $qF_{II}$ orthogonal polynomials. 
\begin{definition}[$qF_{II}$ RHP]\label{quartic RHP}
Let $\Gamma$ be an appropriate curve (see Definition \ref{admissable}) with interior $\mathcal D_-$ and exterior $\mathcal D_+$. A $2\times 2$ complex matrix function $Y_n(z)$, $z\in\mathbb C$, is a solution of the $qF_{II}$ RHP if it satisfies the following conditions:
\begin{enumerate}[label={{\rm (\roman *)}}]
\begin{subequations}\label{new rhp}
\item $Y_n(z)$ is meromorphic in $\mathbb{C}\setminus \Gamma$,  with simple poles at $z=\pm q^{-k}$ for $1\le k\in\mathbb{N}$.
\item $Y_n(z)$ has continuous boundary values $Y_n^-(s)$ and $Y_n^+(s)$ as $z$ approaches $s\in\Gamma$ from $\mathcal D_-$ and $\mathcal D_+$ respectively, where
    \begin{gather} \label{new jump}
      Y_n^+(s)
     =
      Y_n^-(s)
      \begin{bmatrix}
       1 &
       h_q(s)w(s)\\
       0 &
       1
       \end{bmatrix}, \; s\in \Gamma,
    \end{gather}
    and $w(s)$ is defined in Equation \eqref{w def}.
    
\item The residue at each pole $z=\pm q^{-k}$ for $1\le k\in \mathbb{N}$ is given by
    \begin{gather} \label{res cond}
    \mathrm{Res}(Y_n(\pm q^{-k}))
     =
      \lim_{z \rightarrow \pm q^{-k}} Y_n(z)
      \begin{bmatrix}
       0 &
       (z\mp q^{-k})h_q(z)w(z)  \\
       0 &
       0
       \end{bmatrix}.
    \end{gather}
    
\item $Y_n(z)$ satisfies
    \begin{equation}\label{Yinfty decay}
         Y_n(z)\begin{bmatrix}
       z^{-n} &
       0 \\
       0 &
       z^n
       \end{bmatrix} = I + O\left( \frac{1}{|z|} \right), \text{ as $|z| \rightarrow \infty$} ,
    \end{equation} 
    for $z$ such that $|z \pm q^{-k}| > r$, for all $1\le k\in \mathbb{N}$, for fixed $r>0$.
    
\end{subequations} 
\end{enumerate} 
\end{definition}

\begin{remark}
Note that the matrix $Y_n(\pm q^{-k})$ has poles in its second column for $1\le k\in \mathbb{N}$. Thus, the asymptotic decay does not hold near these poles. This is why, following Equation \eqref{Yinfty decay}, we require the added condition $z$ must be such that $|z \pm q^{-k}| > r$, for all $1\le k\in \mathbb{N}$, for fixed $r>0$.
\end{remark}
We now determine the solution of the $q$-RHP.
\begin{lemma}
The unique solution of the $q$-RHP given by Definition \ref{new rhp} is given by
\begin{gather} \label{RHP sol}
Y_n(z) 
=
\begin{bmatrix}
   P_{n}(z) &
   \oint_{\Gamma}\frac{P_{n}(s)w(s)h_q(s)}{2\pi i (z-s)}ds + \int_{q^{-1}}^\infty \frac{P_n(\pm s)w(\pm s)}{z\mp s} d_qs \\
   \gamma_{n-1}^{-1} P_{n-1}(z) &
   \oint_{\Gamma}\frac{P_{n-1}(s)w(s)h_q(s)}{ 2\pi i (z-s)\gamma_{n-1}}ds + \int_{q^{-1}}^\infty \frac{P_{n-1}(\pm s)w(\pm s)}{(z\mp s)\gamma_{n-1}} d_qs
   \end{bmatrix},
\end{gather}
where $\{P_n(z)\}_{n=0}^\infty$ satisfies the orthogonality condition
\begin{equation*}
    \int_{-\infty}^\infty P_n(s)P_m(s)w(s) d_qs = \gamma_n\delta_{n,m}.
\end{equation*}
\end{lemma}

\begin{proof}
The proof follows along similar lines to \cite[Section 2(a)]{qRHP}, with adjustments needed for the current case where the orthogonality weight is not contained in a compact set in $\mathbb{R}$. 

We show that the second row of $Y_n(z)$ must be given by Equation \eqref{RHP sol}. A similar argument can be carried out for the first row. To declutter notation we will label $Y_n(z)$ as $Y(z)$ for the rest of this proof.

It follows from the asymptotic condition, Equation \eqref{Yinfty decay}, that the $(1,1)$ entry of $Y$ must have leading order $z^{n}$ as $z\rightarrow \infty$. As $Y_{(1,1)}(z)$ is analytic and its jump condition, Equation \eqref{new jump}, is given by the identity we immediately conclude that $Y_{(1,1)}(z)$ is a monic polynomial of degree $n$. Similarly, it follows that $Y_{(2,1)}(z)$ is a polynomial of degree at most $n-1$. We denote $Y_{(2,1)}(z)$ by $Q_{n-1}(z)$.  

Consider the bottom right entry of Equation \eqref{RHP sol}. By the jump condition, Equation \eqref{new jump}, we have
\begin{equation}\label{1d jump}
Y_{(2,2)}^{+}(s) = Q_{n-1}(s)w(s)h_q(s) + Y_{(2,2)}^{-}(s) \,. \end{equation}
If there was no residue condition, Equation \eqref{res cond}, then this scalar equation would be solved by the Cauchy transform
\begin{equation}
    Y_{(2,2)}(z) = \frac{1}{2\pi i}\oint_{\Gamma}\frac{Q_{n-1}(s)w(s)h_q(s)}{z-s}ds \,,
\end{equation} 
which is analytic in $\mathbb{C}\setminus \Gamma$ and satisfies Equation \eqref{1d jump}. The residue condition can be readily resolved by letting 
\begin{equation}\label{y22 solved}
    Y_{(2,2)}(z) = \frac{1}{2\pi i}\oint_{\Gamma}\frac{Q_{n-1}(s)w(s)h_q(s)}{z-s}ds + \int_{q^{-1}}^\infty \frac{Q_{n-1}(\pm s)w(\pm s)}{(z\mp s)} d_qs\,,
\end{equation} 
which satisfies both Equations \eqref{new jump} and \eqref{res cond}.
The only step remaining is to prove the asymptotic condition, Equation \eqref{Yinfty decay}, for $Y_{(2,2)}(z)$. Substituting our expression for $h_q(s)$ into Equation \eqref{y22 solved}, we find
\begin{align*}
    Y_{(2,2)}(z) &= \frac{1}{2\pi i}\oint_{\Gamma}\sum_{k=-\infty}^{\infty} \left( \frac{q^{k}}{s-q^{k}}\frac{Q_{n-1}(s)w(s)}{z-s} + \frac{q^{k}}{s+q^{k}}\frac{Q_{n-1}(s)w(s)}{z-s} \right) ds \\
    &\qquad + \int_{q^{-1}}^\infty \frac{Q_{n-1}(\pm s)w(\pm s)}{(z\mp s)} d_qs.
\end{align*} 
which by Cauchy's integral formula, for $z\in \text{ext}(\Gamma)$, becomes
\begin{eqnarray*}
Y_{(2,2)}(z) &=& \sum_{k=0}^{\infty} \left( q^{k}\frac{Q_{n-1}(q^{k})w(q^{k})}{z-q^{k}} + q^{k}\frac{Q_{n-1}(-q^{k})w(-q^{k})}{z+q^{k}} \right) \\
&&\,+ \int_{q^{-1}}^\infty \frac{Q_{n-1}(\pm s)w(\pm s)}{(z\mp s)} d_qs, \\
&=& \int^{\infty}_{-\infty} \frac{Q_{n-1}(x)w(x)}{z-x} d_{q}x \,,
\end{eqnarray*}
where the sum to infinity is well defined on $\Gamma$, as $h_q(s)$ converges as $k\rightarrow \infty$, and the Jackson integral of an analytic function is well defined. Using the geometric series with remainder
\begin{equation}\label{exp}
    \frac{1}{z-x} = \sum^{l}_{k=0}\left( \frac{x^{k}}{z^{k+1}} \right) + \frac{x^{l+1}}{z^{l+1}(z-x)}\,, \quad \text{for}\,x\neq z \,,
\end{equation}
we find
\begin{eqnarray*}
Y_{(2,2)}(z) &=& \int^{\infty}_{-\infty} \frac{Q_{n-1}(x)w(x)x^n}{z^{n}(z-x)} d_{q}x \\
&& + \sum^{n-1}_{k=0} \frac{1}{z^{k+1}}\int^{\infty}_{-\infty} Q_{n-1}(x)w(x)x^k d_{q}x  \,.
\end{eqnarray*} 
Note that the asymptotic condition, Equation \eqref{Yinfty decay}, holds when the last term on the RHS is zero for $k = 0,1,2,...,n-2$. This is true iff 
\[ \int^{\infty}_{-\infty} Q_{n-1}(x)w(x)x^{k} d_{q}x  = 0\,, \quad \text{for} \, k\leq n-2 \,, \]
which is satisfied when $Q_{n-1}$ is an orthogonal polynomial of degree $n-1$ on the $q$-lattice with respect to the weight $w(x)$. This is the class of $qF_{II}$ polynomials. We conclude that the solution of $Y_{(2,2)}(z)$ is given by
\[
    Y_{(2,2)}(z) = \left\{\begin{array}{lr}
       \frac{1}{\gamma_{n-1}}\int^{\infty}_{-\infty} \frac{P_{n-1}(x)w(x)}{z-x} d_{q}x -\frac{1}{\gamma_{n-1}} P_{n-1}(z)w(z)h_q(z), & \text{for } z \in \text{int}(\Gamma),\\
        \frac{1}{\gamma_{n-1}}\int^{\infty}_{-\infty} \frac{P_{n-1}(x)w(x)}{z-x} d_{q}x, & \text{for } z \in \text{ext}(\Gamma).
        \end{array}\right.
\]
After appropriate scaling, and repeating the same arguments for the first row, it follows that Equation \eqref{RHP sol} is a solution of the $q$-RHP given by Definition \ref{quartic RHP}.\\ 

Uniqueness of this solution follows from consideration of the determinant. Observe that the jump matrix $J = Y_{-}^{-1}Y_{+}$ satisfies $\text{det}(J) = 1$. It immediately follows that $\text{det}(Y^{+}) = \text{det}(Y^{-})$ on $\Gamma$. Furthermore, by the residue condition, Equation \eqref{res cond}, $\text{det}(Y)$ has no poles. Thus, $\text{det}(Y)$ is an entire function. By the asymptotic condition, Equation \eqref{Yinfty decay}, $\text{det}(Y) \rightarrow 1$, and so by Liouville's theorem, it follows that $\text{det}(Y) = 1$. This implies $Y^{-1}$ exists and is meromorphic in $\mathbb{C}/\Gamma$. 

Now suppose that there exists a second solution of the $q$-RHP and denote this solution by $\widehat{Y}$. If we define $M = \widehat{Y}Y^{-1}$, it follows that the jump conditions (and residue conditions) effectively cancel and $M_{+} = M_{-}$. Thus, $M$ is entire and $M \rightarrow I$ as $z\rightarrow \infty$. Hence, by Liouville's theorem $M = I$. We conclude that $\widehat{Y} = Y$ and, therefore, there is a single unique solution of the $q$-RHP.
\end{proof}

\begin{remark}\label{Lax remark}
It can be shown that solution of the RHP given by Definition \ref{quartic RHP} satisfies a Lax pair, namely a $q$-difference equation in $z$ and a recurrence relation for the parameter $n$. The proof follows using the same arguments as in \cite{qRHP}.
\end{remark}

\begin{remark}\label{general c remark}
One can repeat the arguments above to show that there is unique solution of Definition \ref{quartic RHP} with a modified function 
\[ h_q(z) \to c^{-1}h_q(z/c) ,\]
for some real constant $c$, and that it is given by
\begin{gather} \nonumber
Y^{(c)}_n(z) 
=
\begin{bmatrix}
   P^{(c)}_{n}(z) &
   \oint_{\Gamma}\frac{P^{(c)}_{n}(s)w(s)h_q(s/c)}{2c\pi i (z-s)}ds + \int_{q^{-1}}^\infty \frac{P^{(c)}_n(\pm cs)w(\pm cs)}{z\mp cs} d_qs \\
   \gamma_{n-1}^{-1} P^{(c)}_{n-1}(z) &
   \oint_{\Gamma}\frac{P^{(c)}_{n-1}(s)w(s)h_q(s/c)}{ 2c\pi i (z-s)\gamma^{(c)}_{n-1}}ds + \int_{q^{-1}}^\infty \frac{P^{(c)}_{n-1}(\pm cs)w(\pm cs)}{(z\mp cs)\gamma^{(c)}_{n-1}} d_qs
   \end{bmatrix},
\end{gather}
where $\{P^{(c)}_n(z)\}_{n=0}^\infty$ satisfies the orthogonality condition
\begin{equation*}
    \int_{-\infty}^\infty P^{(c)}_n(cs)P^{(c)}_m(cs)w(cs) d_qs = \gamma^{(c)}_n\delta_{n,m}.
\end{equation*}

\end{remark}

\section{Transformations of RHP}\label{trans section}
In this section we will transform the RHP given in Definition \ref{quartic RHP} to more easily determine the asymptotics of $Y_n(z)$ as $n\to\infty$. Throughout this section we will assume that $n$ is even. First, we introduce some new functions which will be used when transforming the RHP. We show that these functions satisfy certain difference equations. 
\begin{definition}
Define $g:\mathbb{C}\setminus \{0\} \to \mathbb{C}$ as
\begin{equation}\label{g def}
    g(z) = (z^2;q^2)_\infty (q^2z^{-2};q^2)_\infty.
\end{equation}
\end{definition}
\begin{lemma}\label{g diff lemma} 
$g(z)$ satisfies the difference equation
\begin{equation} \label{g diff}
g(qz) = -z^{-2}g(z) .
\end{equation}
\end{lemma}
\begin{proof}
Substituting $qz$ into Equation \eqref{g def} we find
\begin{eqnarray*}
g(qz) &=& (q^2z^2;q^2)_\infty (z^{-2};q^2)_\infty ,\\
    &=& \prod_{j=0}^\infty (1-q^2z^2q^{2j}) \prod_{j=0}^\infty (1-z^{-2}q^{2j}), \\
    &=& \frac{1-z^{-2}}{1-z^2} \prod_{j=0}^\infty (1-z^2q^{2j}) \prod_{j=0}^\infty (1-q^2z^{-2}q^{2j}) ,\\
    &=& \frac{1-z^{-2}}{1-z^2} g(z), \\
    &=& -z^{-2} g(z) .
\end{eqnarray*}
\end{proof}
\begin{remark}
By induction using Equation \eqref{g diff} we find that for even $n$
\begin{equation}\label{g diff induction}
    g(q^{-n/2}z) = (-1)^{n/2}q^{-n^2/4-n/2}z^n g(z).
\end{equation}
\end{remark}
\begin{definition}
Define $\omega:\mathbb{C}\setminus (\{0\}\cup \{e^{\frac{i\pi(1+2n)}{4}} q^k\}_{k=-\infty}^\infty) \to \mathbb{C}$, where $n = 0,1,2,3$, as
\begin{equation}\label{omega def}
    \omega(z) = 1/(-z^4,-q^{4}z^{-4};q^4)_\infty.
\end{equation}
\end{definition}

\begin{lemma}\label{omega diff lemma}
$\omega(z)$ satisfies the difference equation
\begin{equation} \label{omega diff}
\omega(qz) = z^{4}\omega(z) .
\end{equation}
\end{lemma}
\begin{proof}
The proof follows from definition of $\omega(z)$ (by applying the same arguments as in Lemma \ref{g diff lemma}).
\end{proof}

\begin{lemma}
The function $w(z)$ defined in Equation \eqref{w def} satisfies the difference equation
\begin{equation}\label{w diff}
     w(qz) = (1+z^4)w(z).
\end{equation}
\end{lemma}
\begin{proof}
The proof follows from definition of $w(z)$ (by applying the same arguments as in Lemma \ref{g diff lemma}).
\end{proof}
\begin{remark}
By induction using Equation \eqref{w diff} we find that for even $n$
\begin{equation}\label{w diff induction}
    w(q^{-n/2}z) = z^{-2n}q^{n(\frac{n}{2}+1)}(-q^{2n+4}z^{-4};q^4)_\infty\omega(z).
\end{equation}
Taking $z\to zq^{n/2}$, Equation \eqref{w diff induction} gives
\begin{equation}\label{w omega relation}
    w(z) = z^{-2n}q^{n(-\frac{n}{2}+1)}(-z^{-4}q^4;q^4)_\infty\omega(zq^{n/2}).
\end{equation}
\end{remark}

\subsection{RHP transformations}\label{RHP transformations}
Before proceeding we introduce some notation. Consider the contour $\Gamma$, scaled such that the modulus of points on it are multiplied by $q^{-n/2}$. We denote this new contour by $\Gamma_{q^{-n/2}}$. (If $\Gamma$ were the unit circle, $\Gamma_{q^{-n/2}}$ would be a circle with radius $q^{-n/2}$.) Consider the following transformation to the RHP given by Definition \ref{quartic RHP}:

 \begin{align} \nonumber
     U_n(z)= \left\{\begin{array}{lr}
     Y_n(z), &\text{for}\,z\in\mathcal{D}_-, \\
     Y_{n}(z)
      \begin{bmatrix}
       g(z)^{-1} &
       0\\
       0 &
       g(z)
       \end{bmatrix},  &\text{for}\,z\in\mathcal{D}_{-,q^{-n/2}}\setminus\mathcal{D}_{-},\\
    Y_{n}(z) \begin{bmatrix}
       z^{-n} &
       0 \\
       0 &
       z^n
       \end{bmatrix}, &\text{for}\,z\in\mathcal{D}_{+,q^{-n/2}}.
    \end{array}\right.
    \end{align}
    
This gives a new RHP for $U_n$ with two jumps (and also some more poles which will be discussed shortly). At $z \in \Gamma_{q^{-n/2}}$ we apply Equation \eqref{g diff induction} to determine
\begin{equation}\label{eq 3.9}
 g(q^{-n/2}e^{i\theta})(q^{-n/2} e^{i\theta})^{-n} = g(e^{i\theta}) (-1)^{n/2}q^{n^2/4-n/2}.
\end{equation}  
Motivated by Equation \eqref{eq 3.9} we make the transformation
\begin{align}\label{wnz def} 
 W_n(z)
 = \left\{\begin{array}{lr}
 \begin{bmatrix}
   c_n^{n} &
   0 \\
   0 &
   c_n^{-n} 
   \end{bmatrix}
 U_n(z)
  \begin{bmatrix}
   c_n^{-n} &
   0 \\
   0 &
   c_n^n
   \end{bmatrix}, &\text{for}\,z\in\mathcal{D}_{+,q^{-n/2}}, \\
 \begin{bmatrix}
   c_n^{n} &
   0 \\
   0 &
   c_n^{-n} 
   \end{bmatrix}
  U_n(z),\qquad &\text{for}\, z\in\mathcal{D}_{-,q^{-n/2}},
  \end{array}\right.
\end{align}
where 
\[ c_n = (-1)^{1/2}q^{n/4-1/2}.\]
This new matrix $W_n(z)$ satisfies the following RHP:
\begin{definition}[$W_n$-RHP]\label{W_n RHP}
Let $\Gamma$ be an appropriate curve (see Definition \ref{admissable}) with interior $\mathcal D_-$ and exterior $\mathcal D_+$. A $2\times 2$ complex matrix function $W_n(z)$, $z\in\mathbb C$, is a solution of the $W_n$ RHP if it satisfies the following conditions:
\begin{enumerate}[label={{\rm (\roman *)}}]
\begin{subequations}
\item $W_n(z)$ is meromorphic in $\mathbb{C}\setminus( \Gamma\bigcup \Gamma_{q^{-n/2}})$,  with simple poles at $z=\pm q^{-k}$ for $1\le k\in\mathbb{N}$.
\item $W_n(z)$ has continuous boundary values $W_n^-(s)$ and $W_n^+(s)$ as $z$ approaches $s\in\Gamma$ from $\mathcal D_{-}$ and $\mathcal D_{+}$ respectively, where 
    \begin{gather} \label{first jump}
      W_n^+(s)
     =
      W_n^-(s)
      \begin{bmatrix}
        g(s)^{-1} &
        g(s)w(s)h_q(s)  \\
        0 &
        g(s)
        \end{bmatrix}, \; s\in\Gamma.
    \end{gather}
    
\item $W_n(z)$ has continuous boundary values $W_n^-(s)$ and $W_n^+(s)$ as $z$ approaches $s\in\Gamma_{q^{-n/2}}$ from $\mathcal D_{-,q^{-n/2}}$ and $\mathcal D_{+,q^{-n/2}}$ respectively, where 
    \begin{gather} \label{second jump}
      W_n^+(s)
     =
      W_n^-(s)
      \begin{bmatrix}
        g(sq^{n/2}) &
        0  \\
        0 &
        g(sq^{n/2})^{-1}
        \end{bmatrix}, \; s\in\Gamma_{q^{-n/2}}.
    \end{gather}

\item $W_n(z)$ satisfies
    \begin{equation}
    W_n(z) = I + O\left( \frac{1}{|z|} \right), \text{ as $|z| \rightarrow \infty$} .
    \end{equation} 
    Note that $W_n(\pm q^{-k})$ has poles in the second column for $1\le k\in \mathbb{N}$. Thus, the decay condition does not hold near these poles. For example: the decay condition holds for $z$ such that $|z \pm q^{-k}| > r$, for all integer $k>n/2$, for fixed $r>0$.
    
\item The residue at the poles $z=\pm q^{-k}$ for $1 \leq k \leq n/2$ is given by
    \begin{gather} 
    \text{Res}(W_n(\pm q^{-k}))
     =
      \lim_{z \rightarrow \pm q^{-k}} W_n(z)
      \begin{bmatrix}
       0 &
       0  \\
       (z\mp q^{-k})g(z)^{-2}h_q(z)^{-1}w(z)^{-1} &
       0
       \end{bmatrix}.
    \end{gather}
    
\item The residue at the poles $z=\pm q^{-k}$ for $k>n/2$ is given by
    \begin{gather} \label{wn res}
    \text{Res}(W_n(\pm q^{-k}))
     =
      \lim_{z \rightarrow \pm q^{-k}} W_n(z)
      \begin{bmatrix}
       0 &
       (z\mp q^{-k})z^{2n}h_q(z)w(z)c_n^{2n}  \\
       0 &
       0
       \end{bmatrix}.
    \end{gather}

\end{subequations} 
\end{enumerate} 
\end{definition}

\section{Near-field RHP}\label{near-field rhp section}
We will show that the solution $W_n(z)$ of the RHP given in Definition \ref{W_n RHP} approaches a limiting solution $G(z)$. To do this, we are going to solve two separate RHPs, which we will call the near-field RHP and the far-field RHP. These RHPs will be chosen to mimic the two jump conditions satisfied by $W_n(z)$ at $\Gamma$ and $\Gamma_{q^{-n/2}}$ respectively. This section is devoted to the solution of the near-field RHP.

Motivated by the form of Equation \eqref{first jump}, we first introduce the following RHP.

\begin{definition}[$\mathfrak{W}$-RHP]\label{near model RHP}
Let $\Gamma$ be an appropriate curve (see Definition \ref{admissable}) with interior $\mathcal D_-$ and exterior $\mathcal D_+$. A $2\times 2$ complex matrix function $\mathfrak{W}(z)$, $z\in\mathbb C$, is a solution of the $\mathfrak{W}$-RHP if it satisfies the following conditions:
\begin{enumerate}[label={{\rm (\roman *)}}]
\begin{subequations}
\item $\mathfrak{W}(z)$ is meromorphic in $\mathbb{C}\setminus \Gamma$,  with simple poles in the first column at $z=\pm q^{-k}$ for $1\le k\in \mathbb{N}$.
\item $\mathfrak{W}(z)$ has continuous boundary values $\mathfrak{W}^-(s)$ and $\mathfrak{W}^+(s)$ as $z$ approaches $s\in\Gamma$ from $\mathcal D_-$ and $\mathcal D_+$, and 
\begin{gather} \label{eq:gwh}
  \mathfrak{W}^+(s)
 =
  \mathfrak{W}^-(s)
  \begin{bmatrix}
    g(s)^{-1} &
    g(s)w(s)h_q(s)  \\
    0 &
    g(s)
    \end{bmatrix}, \; s\in\Gamma,
\end{gather}
where $w(s)$ and $g(s)$ are defined in Equations \eqref{w def} and \eqref{g def} respectively.
    
\item $\mathfrak{W}(z)$ satisfies
    \begin{gather} 
    \mathfrak{W}(z) = 
      \begin{bmatrix}
       1 &
       0  \\
       c_0h_q(z)&
       1
       \end{bmatrix} + O\left(\frac{1}{z}\right),
    \end{gather}
    where $c_0$ is a non-zero constant. Due to the simple poles in the first column of $\mathfrak{W}(z)$, the asymptotic decay condition only holds for $|z\pm q^{-k}| > q^{-k/R}$ for any $R>0$.
    
\item The residue at the poles $z=\pm q^{-k}$ for $1\le k\in \mathbb{N}$ is given by
    \begin{gather} 
    \text{Res}(\mathfrak{W}(\pm q^{-k}))
     =
      \lim_{z \rightarrow \pm q^{-k}} \mathfrak{W}(z)
      \begin{bmatrix}
       0 &
       0  \\
       (z\mp q^{-k})g(z)^{-2} h_q(z)^{-1}w(z)^{-1} &
       0
       \end{bmatrix}.
    \end{gather}
\end{subequations} 
\end{enumerate} 
\end{definition}
To solve this RHP, a series of Lemmas are required.
\begin{lemma}\label{L01}
Consider the difference equation
\begin{equation}\label{near 0 q diff}
    y(q^{-2}z) + \bigl(q^{-3}z^2(1+q^{-1})-(1+q^{-1})\bigr)y(q^{-1}z) + q^{-1}(1+q^{-4}z^4)y(z)=0.
\end{equation}
There exists two entire solutions of Equation \eqref{near 0 q diff}, one even and one odd. 
\end{lemma}
\begin{proof}
Let 
\[y(z) = \sum_{i=0}^\infty y_iz^i .\]
Substituting this into Equation \eqref{near 0 q diff} and comparing coefficients of $z$, we find that $y(z)$ is a solution iff
\[y_i = -\frac{(1+q^{-1})q^{-i}y_{i-2} + q^{-4}y_{i-4}}{q^{-2i+1}-(1+q)q^{-i}+1} .\]
Lemma \ref{L01} follows immediately.
\end{proof}
\begin{definition}
Following from Lemma \ref{L01} we define $a(z)$ as the even entire solution to Equation \eqref{near 0 q diff} normalised such that its first non-zero Taylor series coefficient $(\mathrm{at}\,O(1))$ is unity. Similarly, we define $b(z)$ as the odd entire solution normalised such that its first non-zero Taylor series coefficient $(\mathrm{at}\,O(z))$ is unity.
\end{definition}

Motivated by the $(1,2)$ entry of the right side of Equation \eqref{eq:gwh}, we consider the properties of the product $y(z)g(z)w(z)$.
\begin{lemma}\label{v lemma}
Define 
\[ v(z) = y(z)g(z)w(z) ,\]
then $v(z)$ is a solution of the difference equation
\begin{equation}\label{v diff}
    (q^6z^{-4}+q^{-2})v(q^{-2}z) + (z^{-2}q(1+q)-q^{-1}(1+q^{-1}))v(q^{-1}z) + q^{-1}v(z) = 0.
\end{equation}
Furthermore, there exists two solutions to Equation \eqref{v diff} analytic in $\mathbb{C}\setminus\{0\}$ which can be represented by an even and odd power series at infinity. 
\end{lemma}
\begin{proof}
From the definition of $v(z)$ and Lemmas \ref{g diff} and \ref{w diff} we find that
\begin{eqnarray*}
    y(q^{-1}z) &=& \frac{v(q^{-1}z)}{g(q^{-1}z)w(q^{-1}z)},\\
    &=& -\frac{q^2z^{-2}(1+q^{-4}z^4)v(q^{-1}z)}{g(z)w(z)}.
\end{eqnarray*}
Substituting the above into Equation \eqref{near 0 q diff} we determine that $v(z)$ satisfies the difference equation
\begin{equation}\nonumber
    (q^6z^{-4}+q^{-2})v(q^{-2}z) - (q^{-1}(1+q^{-1})-(q^2+q)z^{-2})v(q^{-1}z) + q^{-1}v(z)=0.
\end{equation}
Let $v(z) = \sum_{i=0}^\infty v_iz^{-i}$, such a power series is a solution of Equation \eqref{v diff} iff
\begin{equation}\nonumber
    v_i = -\frac{q(1+q)q^{i}v_{i-2} + q^{2i}v_{i-4}}{q-(1+q)q^i + q^{2i}}.
\end{equation}
Lemma \ref{v lemma} follows immediately (note that the power series for $v(z)$ converges everywhere).
\end{proof}
\begin{definition}
Following from Lemma \ref{v lemma} we define $\phi_{even}(z)$ for $z\in\mathbb{C}\setminus\{0\}$, as the even solution to Equation \eqref{v diff} normalised such that its first non-zero series coefficient $(\mathrm{at}\,O(1))$ is unity. Similarly, we define $\phi_{odd}(z)$ as the odd solution normalised such that its first non-zero series coefficient $(\mathrm{at}\,O(z^{-1}))$ is unity.
\end{definition}
Motivated by the $(1,1)$ entry of the right side of Equation \eqref{eq:gwh} for $z\in\mathcal{D}_+$, we consider the properties of the product $y(z)g(z)^{-1}$.
\begin{lemma}\label{u lemma}
Define 
\[ u(z) = y(z)g(z)^{-1} ,\]
then $u(z)$ is a solution of the difference equation
\begin{equation}\label{u diff}
    q^{-5}u(q^{-2}z) + q^{-1}(z^{-2}(1+q^{-1})-q^{-3}(1+q^{-1}))u(q^{-1}z) + (z^{-4}+q^{-4})u(z) = 0.
\end{equation}
Furthermore, there exists two solutions to Equation \eqref{u diff} holomorphic for $|z|>q$ which can be represented by an even and odd power series at infinity. 
\end{lemma}
\begin{proof}
From the definition of $u(z)$ and Equation \eqref{g diff} we find that
\begin{eqnarray*}
    y(q^{-1}z) &=& u(q^{-1}z)g(q^{-1}z),\\
    &=& -q^{-2}z^{2}u(q^{-1}z)g(z).
\end{eqnarray*}
Substituting the above into Equation \eqref{near 0 q diff} we determine that $u(z)$ satisfies the difference equation
\begin{equation}\nonumber
    q^{-6}z^{4}u(q^{-2}z) - q^{-2}z^{2}(q^{-3}z^2(1+q^{-1})-(1+q^{-1}))u(q^{-1}z) + q^{-1}(1+q^{-4}z^4)u(z)=0,
\end{equation}
which one can readily show is equivalent to Equation \eqref{u diff}. Let $u(z) = \sum_{i=0}^\infty u_iz^{-i}$, such a power series is a solution of Equation \eqref{u diff} iff
\begin{equation}\label{u recur}
    u_i = -\frac{(1+q)q^{i+3}u_{i-2} + q^{5}u_{i-4}}{q-(1+q)q^i + q^{2i}}.
\end{equation}
For large index $i$, we can deduce from Equation \eqref{u recur} that
\begin{equation}\nonumber
    \mathrm{max}(|u_i|,|u_{i+2}|) = (q^4 + O(q^i))\mathrm{max}(|u_{i-4}|,|u_{i-2}|).
\end{equation}
Taking a telescopic product we conclude that the sum $\sum_{i=0}^\infty u_iz^{-i}$ converges if $|z|>q$. Lemma \ref{v lemma} follows immediately.
\end{proof}

\begin{definition}
Following from Lemma \ref{u lemma} we define $\varphi_{even}(z)$ as the even solution to Equation \eqref{u diff}, analytic in $|z|>q$, normalised such that its first non-zero series coefficient $(\mathrm{at}\,O(1))$ is unity. Similarly, we define $\varphi_{odd}(z)$ as the odd solution, analytic in $|z|>q$, normalised such that its first non-zero series coefficient $(\mathrm{at}\,O(z^{-1}))$ is unity.
\end{definition}

\begin{lemma}\label{a/g lemma}
The function $a(z)g(z)^{-1}$, where $a(z)$ is defined in Lemma \ref{L01}, can be written as
\[ a(z)g(z)^{-1} = \eta_1 h_q(z)\varphi_{odd}(z) + \eta_2 \varphi_{even}(z), \]
where $\eta_1,\eta_2$ are constants. Similarly $b(z)g(z)^{-1}$ can be written as
\[ b(z)g(z)^{-1} = \eta_3\varphi_{odd}(z) + \eta_4 h_q(z)\varphi_{even}(z). \]
\end{lemma}
\begin{proof}
From Lemma \ref{u lemma}, we conclude that $a(z)$ satisfies the same difference equation as $g(z)\varphi_{even}(z)$ and $g(z)\varphi_{odd}(z)$. It follows that
\[ a(z) = g(z)(C_1(z)\varphi_{even}(z) + C_2(z)\varphi_{odd}(z)), \]
for some functions $C_i(z)$ which satisfy $C_i(qz) = C_i(z)$. As $a(z)$ is analytic everywhere and $\varphi_{even}(z)$ and $\varphi_{odd}(z)$ are holomorphic for $|z|>q$ we conclude that $C_i(z)$ is constant or has simple poles at the zeros of $g(z)$, which occur at $\pm q^{k}$ for $k\in\mathbb{Z}$. Applying Corollary \ref{even coror} to $C_1(z)$ and $C_2(z)$, and comparing even and odd terms we conclude that 
\[ a(z) = g(z)(\eta_1 h_q(z)\varphi_{odd}(z) + \eta_2 \varphi_{even}(z)), \]
and the first part of Lemma \ref{a/g lemma} follows immediately. The equation for $b(z)/g(z)$ also follows using similar arguments.
\end{proof}

\begin{remark}
It follows from the above Lemma that $\eta_1$, $\eta_2$, $\eta_3$ and $\eta_4$ are all non-zero. We know that $a(z)$ is a non-zero analytic function and thus $a(q^k)\neq 0$ for large enough $k$. However, as shown in Corollary \ref{corollary a.4}, $g(q^k)=h_q(q^{k+1/2}) = 0$ for $k\in\mathbb{Z}$. Thus, if $\eta_1=0$ or $\eta_2=0$ then we arrive at a contradiction. (Note that by Equation \eqref{u diff} $\varphi(z)$ cannot have poles along the real axis.)  
\end{remark}

\begin{lemma}\label{agw lemma}
$\phi_{even}(z)$ as defined in Lemma \ref{v lemma} can be written as
\[ \phi_{even}(z) = g(z)w(z)(\lambda_3 h_q(z)b(z)+\lambda_4 a(z) ),\]
where $\lambda_3$, $\lambda_4$ are constants. Similarly 
\[ \phi_{odd}(z) = g(z)w(z)(\lambda_1 h_q(z) a(z) + \lambda_2 b(z)),\]
\end{lemma}
\begin{proof}
From Lemma \ref{v lemma} we conclude that $\phi_{even}(z)$ satisfies the same difference equation (Equation \eqref{v diff}) as $a(z)g(z)w(z)$ and $b(z)g(z)w(z)$. It follows that
\[ \phi_{even}(z) = g(z)w(z)(C_1(z) a(z) + C_2(z) b(z)),\]
for some functions $C_i(z)$ that satisfy $C_i(qz) = C_i(z)$. As $a(z)$ and $b(z)$ are analytic everywhere and $\phi_{even}(z)$ is holomorphic in $\mathbb{C}\setminus \{0\} $ we conclude that $C_i(z)$ is constant or has simple poles at the zeros of $g(z)$, which occur at $\pm q^{k}$ for $k\in\mathbb{Z}$. Applying Corollary \ref{even coror} to $C_1(z)$ and $C_2(z)$, and comparing even and odd terms we conclude that 
\begin{equation}\label{phi_even def}
    \phi_{even}(z) = g(z)w(z)(\lambda_4 a(z) + \lambda_3 h_q(z)b(z)),
\end{equation} 
and the first part of Lemma \ref{a/g lemma} follows immediately. The equation for $\phi_{odd}(z)$ also follows using similar arguments.
\end{proof}
\begin{remark}\label{a b infty relation}
We note that $w(z)$ has poles at $e^{\frac{i(\pi + 2n\pi)}{4}}q^{-k}$ for $k \in \mathbb{N}$, and $n\in 0,1,2,3$. As $\phi_{even}$ is analytic at these locations we conclude 
\[ \lambda_4 a(e^{\frac{i(\pi + 2n\pi)}{4}}q^{-k}) + \lambda_3 h_q(e^{\frac{i(\pi + 2n\pi)}{4}})b(e^{\frac{i(\pi + 2n\pi)}{4}}q^{-k})=0. \]
Thus,
\[ \lambda_4  = -\lambda_3 \frac{h_q(e^{\frac{i\pi}{4}})b(e^{\frac{i\pi}{4}})}{a(e^{\frac{i\pi}{4}})} . \]
\end{remark}

We are now in a position to solve the RHP given in Definition \ref{near model RHP}. 

\begin{lemma}\label{frakW def}
The solution of the $\mathfrak{W}$-RHP (Definition \ref{near model RHP}) is given by:
\begin{align} \nonumber
\mathfrak{W}(z) = \left\{\begin{array}{lr}
\begin{bmatrix}
 \eta_2^{-1}a(z) &
\lambda_2\eta_{2}^{-1}\lambda_1^{-1} w(z)b(z)\\
\lambda_3 b(z) &
\lambda_4 w(z)a(z)
\end{bmatrix}, &\text{for}\,z\in\mathcal{D}_-, \\
\begin{bmatrix}
\eta_2^{-1} a(z) &
\lambda_2\eta_{2}^{-1}\lambda_1^{-1} w(z)b(z)\\
\lambda_3 b(z) &
\lambda_4 w(z)a(z)
\end{bmatrix}
\begin{bmatrix}
g(z)^{-1} &
w(z)g(z)h_q(z)\\
0 &
g(z)
\end{bmatrix}, &\text{for}\,z\in\mathcal{D}_{+}.
\end{array}\right.
\end{align}
where,
\begin{eqnarray}
    \lambda_2  &=& -\lambda_1 \frac{h_q(e^{\frac{i\pi}{4}})b(e^{\frac{i\pi}{4}})}{a(e^{\frac{i\pi}{4}})} \label{lambda12},\\
    \lambda_4  &=& -\lambda_3 \frac{h_q(e^{\frac{i\pi}{4}})b(e^{\frac{i\pi}{4}})}{a(e^{\frac{i\pi}{4}})}. \label{lambda34}
\end{eqnarray}

\end{lemma}
\begin{proof}
First we show condition (i) (meromorphicity) is satisfied. With the choice of $\lambda_i$ given by Equations \eqref{lambda12} and \eqref{lambda34} it is clear from Remark \ref{a b infty relation} that $\mathfrak{W}(z)$ has analytic entries in the second column for $z\in \mathcal{D}_+$. In particular the second column has entries equal to $\eta_{2}^{-1}\lambda_1^{-1}\phi_{odd}(z)$ and $\phi_{even}(z)$ for $z\in\mathcal{D}_{+}$. Meromorphicity of the LHS column follows immediately from the definition of $a(z)$ and $b(z)$.\\

Conditions (ii) and (iv) follow immediately from the definition of $\mathfrak{W}(z)$. It is left to show condition (iii) is satisfied.
Taking the limit $z\to \infty$ of Equation \eqref{phi_even def} and applying Lemma \ref{a/g lemma} we find 
\[\lambda_3 \eta_4 h_q(z)^{2} + \lambda_4 \eta_2 \sim  \frac{1}{g(z)^{2}w(z)}.\]
Thus,
\[\lambda_3   =  \lim_{k\to \infty} \frac{1}{\eta_4 g(q^{-k})^{2}w(q^{-k})h_q(q^{-k})^2}.\]
Let
\begin{equation}\label{c0 def}
    c_0 = \lim_{k\to \infty} \frac{1}{g(q^{-k})^{2}w(q^{-k})h_q(q^{-k})^2}.
\end{equation}
Hence, applying Lemma \ref{a/g lemma} again we find that as $z\to\infty$ 
\[\lambda_3 b(z)g(z)^{-1} \sim  h_q(z) c_0.\]
We conclude that as $z\to\infty$ the matrix $\mathfrak{W}(z)$ behaves like
\begin{eqnarray} 
\mathfrak{W}(z) &=& \begin{bmatrix}
\eta_2^{-1}a(z) &
\lambda_2\eta_{2}^{-1}\lambda_1^{-1} w(z)b(z)\\
\lambda_3 b(z) &
\lambda_4 w(z)a(z)
\end{bmatrix}
\begin{bmatrix}
g(z)^{-1} &
w(z)g(z)h_q(z)\\
0 &
g(z)
\end{bmatrix} \nonumber \\
&=& 
\begin{bmatrix}
1 &
0\\
h_q(z)c_0 &
1
\end{bmatrix} + O(1/z)\label{frak w to 0}
\end{eqnarray}
\end{proof}

\section{Far-field RHP}\label{far-field RHP section}
In this section, we solve the far-field RHP, which we denote by $\mathcal{W}$-RHP (see Definition \ref{far RHP}). The independent variable in the far-field and near-field RHPs are related through a scaling transformation. To distinguish the two, we use $t$ instead of $z$ to denote a complex variable in this section. Motivated by the form of Equation \eqref{second jump} we introduce the following RHP.
\begin{definition}[$\mathcal{W}$-RHP]\label{far RHP}
Let $\Gamma$ be an appropriate curve (see Definition \ref{admissable}) with interior $\mathcal D_-$ and exterior $\mathcal D_+$. A $2\times 2$ complex matrix function $\mathcal{W}(t)$, $t\in\mathbb C$, is a solution of the RHP if it satisfies the following conditions:
\begin{enumerate}[label={{\rm (\roman *)}}]
\begin{subequations}
\item $\mathcal{W}(t)$ is meromorphic in $\mathbb{C}\setminus \Gamma$,  with simple poles at $t=\pm q^{k}$ for $k\in \mathbb{Z}$.
\item $\mathcal{W}(t)$ has continuous boundary values $\mathcal{W}^-(s)$ and $\mathcal{W}^+(s)$ as $t$ approaches $s\in\Gamma$ from $\mathcal D_{-}$ and $\mathcal D_{+}$ respectively, where 
    \begin{gather} 
      \mathcal{W}^+(s)
     =
      \mathcal{W}^-(s)
      \begin{bmatrix}
        g(s) &
        0  \\
        0 &
        g(s)^{-1}
        \end{bmatrix}, \; s\in\Gamma.
    \end{gather}
    
\item $\mathcal{W}(t)$ satisfies
    \begin{equation}
    \mathcal{W}(t) = I + O\left( \frac{1}{|t|} \right), \text{ as $|t| \rightarrow \infty$} .
    \end{equation} 
    Due to the simple poles in the second column of $\mathcal{W}(t)$, the asymptotic decay condition only holds for $|z\pm q^{-k}| > R$ for any $R>0$.
 
\item The residue at the poles $t=\pm q^{k}$ for $k\in \mathbb{N}$ is given by
    \begin{gather} 
    \text{Res}(\mathcal{W}(\pm q^{-k}))
     =
      \lim_{t \rightarrow \pm q^{-k}} \mathcal{W}(t)
      \begin{bmatrix}
       0 &
       0  \\
       (t\mp q^{-k})g(t)^{-2}h_q(t)^{-1}\omega(t)^{-1} &
       0
       \end{bmatrix},
    \end{gather}
    where $\omega(t)$ is defined in Equation \eqref{omega def}.
    
\item The residue at the poles $t=\pm q^{-k}$ for $1\le k\in \mathbb{N}$ is given by
    \begin{gather} \label{w cal res}
    \text{Res}(\mathcal{W}(\pm q^{-k}))
     =
      \lim_{t \rightarrow \pm q^{-k}} \mathcal{W}(t)
      \begin{bmatrix}
       0 &
       (t\mp q^{-k})h_q(t)\omega(t)  \\
       0 &
       0
       \end{bmatrix}.
    \end{gather}
\end{subequations} 
\end{enumerate} 
\end{definition}

We will explicitly solve this RHP, using a similar approach to Section \ref{near-field rhp section}. To do so, we prove a sequence of Lemmas.
\begin{lemma}\label{Linfty1}
Consider the difference equation
\begin{equation}\label{near infty q diff}
    y_1(q^{-2}t)q^7t^{-4} - (1-q^2(q+1)t^{-2})y_1(q^{-1}t) + y_1(t)=0.
\end{equation}
There exists a solution of Equation \eqref{near infty q diff}, analytic in $\mathbb{C}\setminus \{0\}$, which can be represented by the even power series
\[ a_{\infty}(t) = \sum_{j=0}^\infty a_{2j}t^{-2j}, \]
where we take $a_0 = 1$ (w.l.o.g.).

Similarly, there exists a solution analytic in $\mathbb{C}\setminus \{0\}$, of the difference equation
\begin{equation}\label{near infty q diff 2}
    y_2(q^{-2}t)q^7t^{-4} - (q^{-1}-q^2(q+1)t^{-2})y_2(q^{-1}t) + y_2(t)=0,
\end{equation}
which can be represented by the odd power series
\[ b_{\infty}(t) = \sum_{j=0}^\infty b_{2j+1}t^{-2j-1} ,\]
without loss of generality let $b_1 =1$.

\end{lemma}
\begin{proof}
Let 
\[y_1(z) = \sum_{i=0}^\infty y_it^{-i} .\]
Substituting this into Equation \eqref{near infty q diff} and comparing coefficients of $t$, we find that $y(t)$ is a solution iff
\[y_i = -\frac{(1+q)q^{i}y_{i-2} + q^{2i-1}y_{i-4}}{q^i-1} .\]
The first part of Lemma \ref{Linfty1} follows immediately. The second part follows using similar arguments.
\end{proof}

\begin{lemma}\label{v infty lemma}
Define 
\[ \alpha(t) = y_1(t)g(t)\omega(t) ,\]
where $y_1(t)$ satisfies the difference equation given by Equation \eqref{near infty q diff} and $\omega(t)$ is defined in Equation \eqref{omega def}. Then $\alpha(t)$ is a solution of the difference equation
\begin{equation}\label{v infty diff}
q\alpha(q^{-2}t) + (t^{2}q^{-2}-(1+q))\alpha(q^{-1}t) + \alpha(t) = 0.
\end{equation}
Furthermore, there exists two entire solutions of Equation \eqref{v infty diff} which can be represented by an even and odd power series. \\

Similarly, define
\[ \beta(t) = y_2(t)g(t)\omega(t) ,\]
where $y_2(t)$ satisfies the difference equation given by Equation \eqref{near infty q diff 2}, then $\beta(t)$ satisfies the difference equation
\begin{equation}\label{v infty diff2}
q\beta(q^{-2}t) + (t^{2}q^{-3}-(1+q))\beta(q^{-1}t) + \beta(t) = 0.
\end{equation}
Furthermore, there exists two entire solutions which can be represented by an even and odd power series.
\end{lemma}
\begin{proof}
From the definition of $\alpha(t)$, Equations \eqref{g diff} and \eqref{omega diff} we find that
\begin{eqnarray*}
    y_1(q^{-1}t) &=& \frac{\alpha(q^{-1}t)}{g(q^{-1}t)\omega(q^{-1}t)},\\
    &=& -\frac{q^{-2}t^2\alpha(q^{-1}t)}{g(t)w(t)}.
\end{eqnarray*}
Substituting the above into Equation \eqref{near infty q diff} we determine that $\alpha(t)$ satisfies the difference equation
\begin{equation}\nonumber
     \alpha(q^{-2}t)q + (q^{-2}t^2-(q+1))\alpha(q^{-1}t) + \alpha(z)=0.
\end{equation}
Let $\alpha(t) = \sum_{i=0}^\infty \alpha_it^{i}$, such a power series is a solution of Equation \eqref{v infty diff} iff
\begin{equation}\nonumber
    \alpha_i = -\frac{q^{-1-i}\alpha_{k-2}}{1-(1+q)q^{-i} + q^{1-2i}}.
\end{equation}
The first part of Lemma \ref{v infty lemma} follows immediately. The second part follows using similar arguments.
\end{proof}
\begin{definition}
Following from Lemma \ref{v infty lemma} we define $\Psi_{even}(t)$ as the even entire solution to Equation \eqref{v infty diff}, normalised such that its first non-zero Taylor series coefficient $(t^{0})$ is unity. Similarly we define  $\Psi_{odd}(t)$ as the odd entire solution to Equation \eqref{v infty diff}, normalised such that its first non-zero Taylor series coefficient $(t^{1})$ is unity.\\

Furthermore, we define $\varPsi_{even}(t)$ as the even entire solution to Equation \eqref{v infty diff2}, normalised such that its first non-zero Taylor series coefficient $(t^{0})$ is unity. Similarly we define  $\varPsi_{odd}(t)$ as the odd entire solution to Equation \eqref{v infty diff2}, normalised such that its first non-zero Taylor series coefficient $(t^{1})$ is unity.
\end{definition}
\begin{lemma}\label{u infty lemma}
Define 
\[ \Theta_1(t) = y_1(t)g(t)^{-1} ,\]
where $y_1(t)$ is a solution of Equation \eqref{near infty q diff}. Then $\Theta_1(t)$ is a solution of Equation \eqref{v infty diff}.\\

Similarly, $\Theta_2(t) = y_2(t)g(t)^{-1}$ is a solution of Equation \eqref{v infty diff2}.
\end{lemma}
\begin{proof}
From the definition of $\Theta_1(t)$ we find that
\begin{eqnarray}
    \Theta_1(q^{-1}t) &=& \frac{y_1(q^{-1}t)}{g(q^{-1}t)},\nonumber \\
    &=& \frac{\alpha(q^{-1}t)}{g(q^{-1}t)^2w(q^{-1}t)}, \nonumber \\
    &=& \frac{\alpha(q^{-1}t)}{g(t)^2w(t)}, \label{eta diff}
\end{eqnarray}
where we have used Equation \eqref{g diff} and Equation \eqref{omega diff} to arrive at the final line. The first part Lemma \ref{u infty lemma} follows immediately from Equation \eqref{eta diff}. The proof for $\Theta_2(t)$ follows from the same arguments.
\end{proof}

\begin{lemma}\label{a_even/g lemma}
The function $a_{\infty}(t)g(t)^{-1}$, where $a_{\infty}(t)$ is defined in Lemma \ref{Linfty1}, can be written as
\[ a_{\infty}(t)g(t)^{-1} = \mu_1 h_q(t)\Psi_{odd}(t) + \mu_2 \Psi_{even}(t), \]
where $\mu_1,\mu_2$ are a constants. Similarly $a_{odd}(t)g(t)^{-1}$ can be written as
\begin{equation}
     b_{\infty}(t)g(t)^{-1} = \mu_3\varPsi_{odd}(t) + \mu_4 h_q(t)\varPsi_{even}(t). \label{bodd balance}
\end{equation}
\end{lemma}
\begin{proof}
From Lemma \ref{u infty lemma} we conclude that $a_{\infty}(t)$ satisfies the same difference equation (Equation \eqref{near infty q diff}) as $g(t)\Psi_{even}(t)$ and $g(t)\Psi_{odd}(t)$. It follows that
\[ a_{\infty}(t) = g(t)(C_1(t)\Psi_{even}(t) + C_2(t)\Psi_{odd}(t)), \]
for some functions $C_i(t)$ which satisfy $C_i(qt) = C_i(t)$. As $a_{\infty}(t)$ is analytic in $\mathbb{C}\setminus \{0\}$ and $\Psi_{even}(t)$ and $\Psi_{odd}(t)$ are entire we conclude that $C_i(t)$ is constant or has simple poles at the zeros of $g(t)$, which occur at $\pm q^{k}$ for $k\in\mathbb{Z}$. Applying Corollary \ref{even coror} and comparing even and odd terms we conclude that 
\[ a_{\infty}(t) = g(t)(\mu_1 h_q(t)\Psi_{odd}(t) + \mu_2 \Psi_{even}(t)), \]
and the first part of Lemma \ref{a_even/g lemma} follows immediately. The equation for $b_{\infty}(t)g(t)^{-1}$ also follows using similar arguments.
\end{proof}

\begin{lemma}\label{res comparison}
Let $b_{\infty}(t)$, $\mu_4$ and $\varPsi_{even}(t)$ be defined as in Lemma \ref{a_even/g lemma}. Furthermore, define the constant $\mathcal{H}$ as
\begin{equation}\label{calh def}
    \mathcal{H} = \omega(q)g(q)^2h_q(q)^2.
\end{equation}
Then, 
\begin{equation}
\frac{1}{\mu_4\mathcal{H}}b_{\infty}(t)g(t)^{-1} \sim \frac{h_q(t)}{\mathcal{H}}, \label{b/g limit}
\end{equation}
as $t\to 0$. It is also true that 
\begin{equation}
\mathrm{Res}\left(\frac{1}{\mu_4\mathcal{H}}b_{\infty}(q^{k})\omega(q^{k})h_q(q^k)\right) = \mathrm{Res}(\varPsi_{even}(q^k)g(q^k)^{-1}),
\end{equation}
for $k \in \mathbb{Z}$.
\end{lemma}

\begin{proof}
Equation \eqref{b/g limit} follows immediately from taking the limit $t\to 0$ in Equation \eqref{bodd balance}. Multiplying Equation \eqref{bodd balance} by $\omega(t)g(t)h_q(t)$ we find 
\begin{equation}\nonumber
     b_{\infty}(t)\omega(t)h_q(t) = \omega(t)g(t)h_q(t)\mu_3\varPsi_{odd}(t) + \omega(t)g(t)^2h_q(t)^2\mu_4(\varPsi_{even}(t)g(t)^{-1}).
\end{equation}
Studying the residue at $t = q^k$ for $k \in \mathbb{Z}$ we find that 
\begin{eqnarray*}
     \mathrm{Res}(b_{\infty}(q^k)\omega(q^k)h_q(q^k)) &=& \mathrm{Res}\left(\omega(q^k)g(q^k)^2h_q(q^k)^2\mu_4(\varPsi_{even}(q^k)g(q^k)^{-1})\right),\\
     &=& \mathrm{Res}\left(\mathcal{H}\mu_4\varPsi_{even}(q^k)g(q^k)^{-1}\right).
\end{eqnarray*}
\end{proof}

\begin{remark}\label{res comparison a}
Repeating the arguments of Lemma \ref{res comparison} one can readily show
\begin{equation}
     \mathrm{Res}(a_{\infty}(q^k)\omega(q^k)h_q(q^k)) = \mathrm{Res}\left(\mathcal{H}\mu_2\Psi_{even}(q^k)g(q^k)^{-1}\right).
\end{equation}
\end{remark}

We require one last Lemma before determining the solution of the far-field RHP.
\begin{lemma}\label{goes to const}
Let $\varPsi_{even}(t)$ be defined as in Lemma \ref{v infty lemma}, then 
\[ \varPsi_{even}(t)g(t)^{-1} = c_{\varPsi}+O(t^{-1}),\;\mathrm{as}\;t\to\infty ,\]
where $c_\varPsi$ is a non-zero constant and this limit clearly does not hold near the poles of $\varPsi_{even}(t)g(t)^{-1}$, but holds for $t$ satisfying $|t-q^k|>r$, for some fixed $r>0$ and all $k\in\mathbb{Z}$.
\end{lemma}
\begin{proof}
We first show that the residue of the poles of $\varPsi_{even}(t)g(t)^{-1}$ are vanishing faster than $q^{k^2/2}$ as $|k| \to \infty$.\\

Consider the case $k\to +\infty$ ($t\to 0$), from Equation \eqref{g diff induction} we find that
\[ g(q^{n/2}z) = (-1)^{n/2} q^{-n^2/4+n/2}g(z)  .\]
By Lemma \ref{v infty lemma} we know that $\varPsi_{even}(q^k) \sim 1$ as $k \to \infty$. Thus, we conclude 
\[ \mathrm{Res}(\varPsi_{even}(q^k)g(q^k)^{-1}) < O(q^{k^2/2}),\;\text{as}\;k\to+\infty.\]
Note that the above statement is true for $O(q^{ck^2})$, with $c<1$.\\

Consider the case $k\to -\infty$. From Lemma \ref{res comparison} it is clear that a bound on $\mathrm{Res}(b_{\infty}(q^{k})\omega(q^{k})h_q(q^k))$ as $k\to-\infty$ is equivalent to a bound on $\mathrm{Res}(\varPsi_{even}(q^k)g(q^k)^{-1})$. By definition in Lemma \ref{Linfty1} we determine that $b_{\infty}(q^{k}) = O(q^{-k})$. Furthermore, using induction on Equation \eqref{omega diff} we find
\[ \omega(q^kt) = q^{2k(k-1)}t^{4k}w(t).\]
It follows 
\[ \mathrm{Res}(b_{\infty}(q^{k})\omega(q^{k})h_q(q^k)) < O(q^{k^2}),\;\text{as}\;k\to-\infty.\]
Let $R_k$ be the residue of $\varPsi_{even}(t)g(t)^{-1}$ at $t=q^k$. Define the function
\[ F(t) = \varPsi_{even}(t)g(t)^{-1} - \sum_{k=-\infty}^\infty \frac{R_k}{t-q^k}, \]
where this sum is well defined for all $t$ because we have just shown $R_k < O(q^{k^2/2})$ as $|k| \to \infty$. It follows $F(t)$ is holomorphic in $\mathbb{C}\setminus \{0\}$ and can be represented by a Laurent series which converges everywhere. We will show that $F(t) = \sum_{j=0}^\infty F_jt^{-j}$ (i.e. there are no positive powers of $t$). Applying Equation \eqref{v infty diff2} and Equation \eqref{g diff} we find that $v(t) = \varPsi_{even}(t)g(t)^{-1}$ satisfies the difference equation
\begin{equation}
    v(q^{-2}t) + ((1+q)q^3t^{-2}-1)v(q^{-1}t) + q^5t^{-4}v(t) = 0.
\end{equation}
Writing the above in matrix form we have
\begin{gather} \label{matrix equation}
  \begin{bmatrix}
    v(q^{-2}t) \\
    v(q^{-1}t)
    \end{bmatrix}
 =
  \left( \begin{bmatrix}
    1 &
    0  \\
    1 &
    0
    \end{bmatrix} - q^3t^{-2}\begin{bmatrix}
    -(1+q) &
    -q^2t^{-4}  \\
    0 &
    0
    \end{bmatrix} 
    \right) 
    \begin{bmatrix}
    v(q^{-1}t) \\
    v(t)
    \end{bmatrix}.
\end{gather}
Observe that the eigenvalues of the LHS matrix in the above equation are 1 and 0. Hence, repeatedly applying Equation \eqref{matrix equation} to determine the behaviour of $v(t)$ as $t\to \infty$ is essentially a Pochhammer symbol with matrix entries. Thus, $v(t)$, and consequently $F(t)$ are bounded by a constant as $t \to \infty$. We now show that this constant is non-zero. From Lemma \ref{v infty lemma} we know that $\varPsi_{even}(t)$ is an entire function (which is not the constant function), hence $\varPsi_{even}(t)$ must grow in some direction. $\varPsi_{even}(t)$ satisfies Equation \eqref{v infty diff2},
\begin{equation}\nonumber
q\varPsi_{even}(q^{-2}t) + (t^{2}q^{-3}-(1+q))\varPsi_{even}(q^{-1}t) + \varPsi_{even}(t) = 0.
\end{equation}
It follows that as $t$ becomes large there must exist a ray where 
\[ t^{2}q^{-3}\varPsi_{even}(q^{-1}t) \gg \varPsi_{even}(t). \]
Thus, along this ray 
\[ \varPsi_{even}(q^{-2}t) = -t^{2}q^{-3}\varPsi_{even}(q^{-1}t)(1 + O(t^{-2})), \]
and applying Equation \eqref{g diff} we conclude $\varPsi_{even}(t)g(t)^{-1}$ approaches a constant along this ray. Thus, $F(t) = \sum_{j=0}^\infty F_jt^{-j}$ and $F_0 \neq 0$. Lemma \ref{goes to const} follows immediately.
\end{proof}
\begin{remark}\label{Psi/g}
Repeating the same arguments as in Lemma \ref{goes to const} we can conclude 
\begin{equation}
    \Psi_{odd}(t)g(t)^{-1} \to O(t^{-1}), \, \mathrm{as}\,t\to\infty,
\end{equation}
where again this limit clearly does not hold near the poles of $\varPsi_{even}(t)g(t)^{-1}$, but holds for $t$ satisfying $|t-q^k|>r$, for some fixed $r>0$ and all $k\in\mathbb{Z}$.
\end{remark}

Now we are in a position to solve the far-field RHP given in Definition \ref{far RHP}. Let $\mathcal{W}(t)$ be given by
\begin{equation}\label{calW def}
\mathcal{W}(t) = \begin{cases}
\begin{bmatrix}
a_{\infty}(t)g(t)^{-1} &
\mu_2\mathcal{H}\Psi_{odd}(t)\\
\dfrac{b_{\infty}(t)g(t)^{-1}}{\mu_4\mathcal{H}c_\varPsi} &
\varPsi_{even}(t)c_{\varPsi}^{-1}
\end{bmatrix}&\text{for}\,z\in\mathcal{D}_-, \\
&\\
\begin{bmatrix}
a_{\infty}(t) &
\mu_2\mathcal{H} \Psi_{odd}(t)g(t)^{-1}\\
\dfrac{b_{\infty}(t)}{\mu_4\mathcal{H}c_\varPsi} &
\varPsi_{even}(t)g(t)^{-1}c_{\varPsi}^{-1}
\end{bmatrix}, &\text{for}\,z\in\mathcal{D}_{+}.
\end{cases}
\end{equation}

Consider the conditions for this function to solve the far-field RHP. First, we note that condition (i), i.e. meromorphicity, is satisfied as by definition $a_{\infty}$, $b_{\infty}$, $\Psi_{odd}$ and $\varPsi_{even}$ are all analytic in $\mathbb{C}\setminus \{0\}$ (see Lemmas \ref{Linfty1} and \ref{v infty lemma}). 
Second, note that condition (ii), the jump condition, holds by direct calculation. 
Third, to show condition (iii), i.e. asymptotic decay, observe that $\mathcal{W}_{(1,1)}(t)= 1 +O(1/t^2)$ as $t\to\infty$ by the definition of $a_{\infty}(t)$, similarly $\mathcal{W}_{(1,2)}(t)= O(1/t)$ as $t\to\infty$ by the definition of $b_{\infty}(t)$. From Lemma \ref{goes to const} we conclude that $\mathcal{W}_{(2,2)}(t)=1+O(1/t)$ as $t\to\infty$. Similarly from Remark \ref{Psi/g} we conclude $\mathcal{W}_{(2,1)}(t)= O(1/t)$ as $t\to\infty$.
The remaining conditions (iv) and (v), the residue conditions, follow from Lemma \ref{res comparison} and Remark \ref{res comparison a}.
Furthermore, from Lemmas \ref{a_even/g lemma}, \ref{res comparison} and \ref{goes to const} we find
\begin{align}\label{calW to 0}
\mathcal{W}(t) =
\begin{bmatrix}
\mu_2 &
0\\
\dfrac{h_q(t)}{\mathcal{H}c_\varPsi} &
c_{\varPsi}^{-1}
\end{bmatrix} + O(t), \;\text{as}\,t\to0.
\end{align}

\section{Gluing together near- and far-field RHPs}\label{glue section}
We will now glue together the near- and far-field RHPs to approximate the RHP for $W_n(z)$ as $n\to\infty$. The near- and far-field variables in Sections \ref{near-field rhp section} and \ref{far-field RHP section} are related by the linear transformation $t = zq^{n/2}$. \\

We first make a linear transformation to the near-field RHP solution. Let
\begin{gather}  
  \widetilde{\mathfrak{W}}(z)
 =
  \begin{bmatrix}
    \mu_2 &
    0  \\
    0 &
    c_{\varPsi}^{-1}
    \end{bmatrix}\mathfrak{W}(z).
\end{gather}
From Equations \eqref{w def} and \eqref{omega def} one can readily determine that $w(z)\sim\omega(z)$ as $z\to \infty$. Thus, comparing Equations \eqref{c0 def} and \eqref{calh def} we find that $\mathcal{H} = c_0^{-1}$. Hence, applying Equations \eqref{frak w to 0} and \eqref{calW to 0} we find
\[ \lim_{z\to\infty} \widetilde{\mathfrak{W}}(z) = \lim_{t\to 0} \mathcal{W}(t) .\]

We next make a slight modification to $\mathcal{W}(t)$ given by Equation \eqref{calW def}. Note that the residue condition for $W_n(z)$ given in Equation \eqref{wn res} is different to that for $\mathcal{W}(t)$ given in Equation \eqref{w cal res}. To resolve this issue we define the new function
\begin{gather}  
  \widetilde{\mathcal{W}}(t)
 =
  \mathcal{W}(t) - \left(1-\frac{w(z)z^{2n}c_{n}^{2n}}{\omega(t)}\right)\mathcal{W}(t)\begin{bmatrix}
    0 &
    0  \\
    0 &
    1
    \end{bmatrix}.
\end{gather}
Substituting in Equation \eqref{w omega relation} we find that
\begin{gather}  
  \widetilde{\mathcal{W}}(t)
 =
  \mathcal{W}(t) - (1-(-z^{-2};q^4)_\infty)\mathcal{W}(t)\begin{bmatrix}
    0 &
    0  \\
    0 &
    1
    \end{bmatrix}.
\end{gather}
Thus, the difference between $\widetilde{\mathcal{W}}(t)$ and $\mathcal{W}(t)$ is bounded by $O(q^{n/2})$ for $z>q^{-n/4}$ ($t>q^{n/4}$). \\

Define
\begin{align}
G(z) = \left\{\begin{array}{lr}
\widetilde{\mathfrak{W}}(z), &\text{for}\;z\in\mathcal D_{-,q^{-n/4}}, \\
\widetilde{\mathcal{W}}(zq^{n/2}), &\text{for}\;z\in\mathcal D_{+,q^{-n/4}}.
\end{array}\right.
\end{align}
and, furthermore
\begin{equation}\label{R def}
   R(\zeta) = W_n(\zeta q^{-n/4})G(\zeta q^{-n/4})^{-1} .
\end{equation}
Then, $R(\zeta)$ satisfies the following RHP.
\begin{definition}[$R(\zeta)$ RHP]\label{remainder RHP}
A $2\times 2$ complex matrix function $R(\zeta)$, $\zeta\in\mathbb C$, is a solution of the $R(\zeta)$ RHP if it satisfies the following conditions:
\begin{enumerate}[label={{\rm (\roman *)}}]
\begin{subequations}
\item $R(\zeta)$ is analytic in $\mathbb{C}\setminus \Gamma$.
\item $R(\zeta)$ has continuous boundary values $R^-(s)$ and $R^+(s)$ as $\zeta$ approaches $s\in\Gamma$ from $\mathcal D_{-}$ and $\mathcal D_{+}$ respectively, where 
\begin{gather} 
  R^+(s)
 =
  R^-(s)
  \widetilde{\mathfrak{W}}(sq^{-n/4})^{-1}\widetilde{\mathcal{W}}(sq^{n/4})\, \; s\in\Gamma.
\end{gather}
    
\item $R(\zeta)$ satisfies
    \begin{equation}
    R(\zeta) = I + O\left( \frac{1}{|\zeta|} \right), \text{ as $|\zeta| \rightarrow \infty$} .
    \end{equation} 
\end{subequations} 
\end{enumerate} 
\end{definition}
From Equations \eqref{frak w to 0} and \eqref{calW to 0} we find that \[ \| \widetilde{\mathfrak{W}}(sq^{-n/4})^{-1}\widetilde{\mathcal{W}}(sq^{n/4}) - I \|_{\Gamma} = O(q^{n/4}) .\]
Thus, we can apply Theorem \ref{R to 1 theorem} to conclude
\begin{equation}\label{R close to identity theorem}
|R(\zeta)-I| = O(q^{n/4}).
\end{equation}

\section{Proofs of main theorems}\label{main proof section}
Having proved Equation \eqref{R close to identity theorem}, we are now in a position to prove the first two main theorems of this paper.
\begin{proof}[Proof of Theorem \ref{main result 1}]
By the definition of $R(\zeta)$, we find that
\begin{eqnarray*}
    W_n(z)G(z)^{-1} &=& R(zq^{n/4}), \\
    &=& I + O(q^{n/4}), \\
    W_n(z) &=& (I + O(q^{n/4}))G(z).
\end{eqnarray*}
Looking at the $(1,1)$-entry of $W_n(z)$ we find that for $z\leq q^{-n/4}$,
\[ c_n^nP_n(z) = G_{(1,1)}(z) + O(q^{n/4})G_{(2,1)}(z) .\]
Thus, aplying Equation \eqref{frakW def} we find
\[ (-1)^{n/2}q^{\frac{n}{2}(\frac{n}{2}-1)}P_n(z) = \frac{\mu_2}{\eta_2}a(z) + O(q^{n/4})\frac{\lambda_3}{c_\Psi}b(z) .\]
Repeating the arguments above for each matrix entry, Theorem \ref{main result 1} follows immediately.
\end{proof}

\begin{proof}[Proof of Theorem \ref{main result 2}]
Using the transformations detailed in Section \ref{RHP transformations} we find that
\begin{gather}\nonumber
W_n(z)
=
 \begin{bmatrix}
   1 &
   0 \\
   0 &
   1
   \end{bmatrix}
   +
\frac{1}{z}\begin{bmatrix}
    0 &
   \gamma_n c_n^{2n} \\
   \gamma_{n-1}^{-1}c_n^{-2n}  &
   0
   \end{bmatrix} + O(z^{-2}).
\end{gather}
Let,
\begin{gather}\nonumber
\mathcal{W}(t) 
=
    \begin{bmatrix}
   1 &
   0 \\
   0 &
   1
   \end{bmatrix}
   +
\frac{1}{t}\begin{bmatrix}
    0 &
   A \\
   B  &
   0
   \end{bmatrix} + O(t^{-2}).
\end{gather}
Note that there is no difference in the $O(1/t)$ term between $\mathcal{W}(t)$ and $\widetilde{\mathcal{W}}(t)$. Using the definition of $R(\zeta)$ given in Equation \eqref{R def} and Equation \eqref{R close to identity theorem} we find
\begin{eqnarray*}
\gamma_n (-1)^{n/2}q^{\frac{1}{2}(\frac{n}{2}-1)})^{2n} &=& A(1+O(q^{n/2}))q^{-n/2} ,\\
\gamma_n &=& A(1+O(q^{n/2}))q^{-n(\frac{n}{2}-1)}q^{-n/2}, \\
 &=& A(1+O(q^{n/2}))q^{-n(n-1)/2}.
\end{eqnarray*}
Similarly in the bottom left term we find in the limit $n\to\infty$:
\begin{equation*}
    \gamma_{n-1} = q^{\frac{n}{2}(3-n)}\left( B^{-1} + O(q^{n/2})\right).
\end{equation*}
Taking the ratio of $\gamma_n$ and $\gamma_{n-1}$ we find that
\[ \alpha_n = \gamma_n/\gamma_{n-1} = ABq^{-n}(1+O(q^{n/2})).\]
However, we can determine that $AB=q$ by considering the arguments presented in Theorem \ref{theorem diff asym}.
\end{proof}

\section{Recurrence coefficients and $q$-discrete Painlev\'e}\label{recurrence coefficients section}
As discussed in Remark \ref{general c remark}, the class of monic polynomials $\{ P_n^{(c)} \}_{n=0}^\infty$ satisfying the orthogonality condition
\begin{equation}\label{general orthogonality}
    \int^{\infty}_{-\infty} P^{(c)}_n(cx)P^{(c)}_m(cx)w(cx)d_qx = \gamma^{(c)}_n \delta_{n,m}, 
\end{equation} 
where $q< c\leq 1$, satisfy a corresponding RHP. In this section we discuss the connection between the RHP, the asymptotic behaviour of $P_n^{(c)}$ and uniqueness results concerning their recurrence coefficients. First, we use the RHP to show that in general $P_n^{(c)}\neq P_n^{(1)}$.
\begin{lemma}\label{lemma 8.1}
    Let $\{ P_n^{(c)} \}_{n=0}^\infty$ be the class of monic polynomials with orthogonality condition given by Equation \eqref{general orthogonality}. Then, the two classes of orthogonal polynomials corresponding to the cases $c=1$ and $c=q^{1/2}$ are the same. Furthermore,
    \[ \frac{\gamma^{(1)}_1}{\gamma^{(q^{1/2})}_1} = \frac{q^{1/2}h_q(e^{i\pi/4})}{h_q(q^{1/2}e^{i\pi/4})} .\]
    Moreover, if $c\neq 1,q^{1/2}$ then $\{ P_n^{(c)} \}_{n=0}^\infty \neq \{ P_n^{(1)} \}_{n=0}^\infty $.
\end{lemma}

\begin{proof}
Let
\[ \widehat{Y}^{(1)}_{(1,2)}(z) = \int^{\infty}_{-\infty} \frac{P_{n}^{(1)}(x)w(x)}{z-x} d_{q}x - P_{n}^{(1)}(z)w(z)h_q(z). \]
From the arguments in Section \ref{statement of rhp section} it follows that $\widehat{Y}^{(1)}_{(1,2)}(z)$ is meromorphic with simple poles at location of the poles of $w(z)$. At these locations
\begin{equation}\label{res 0.5,1 comparison}
    \mathrm{Res}(\widehat{Y}^{(1)}_{(1,2)}(z)) = - \mathrm{Res}(P_{n}^{(1)}(z)w(z)h_q(z)). 
\end{equation} 
Consider the function
\[ F_n(z) = \frac{h_q(q^{1/2}e^{i\pi/4})}{h_q(e^{i\pi/4})}\widehat{Y}^{(1)}_{(1,2)}(z) + P_{n}^{(1)}(z)w(z)h_q(q^{1/2}z), \]
from Equation \eqref{res 0.5,1 comparison}, Lemma \ref{zero real lemma} and Remark \ref{sym imag} we conclude that $F_n(z)$ is meromorphic with simple poles at $z = q^{k+1/2}$ for $k \in \mathbb{Z}$. The residue of these poles is given by
\[ \mathrm{Res}(F_n(q^{k+1/2})) = P_{n}^{(1)}(q^{k+1/2})w(q^{k+1/2})q^{k+1/2}.\]
Note that $w(z)$ decays much faster than inverse polynomial decay and thus 
\[ \lim_{z\to\infty} F_n(z) = \frac{h_q(q^{1/2}e^{i\pi/4})}{h_q(e^{i\pi/4})} \lim_{z\to\infty} \widehat{Y}^{(1)}_{(1,2)}(z).\]
Hence, the solution for the RHP corresponding to $c=q^{1/2}$ (see Remark \ref{general c remark}) can also be written as
\begin{align}\nonumber
 Y_n^{(q^{1/2})}(z)
 = \left\{\begin{array}{lr}
  \begin{bmatrix}
   P_n^{(1)}(z) &
   q^{-1/2}F_n(z) \\
   \dfrac{P_{n-1}^{(1)}(z)}{\gamma_{n-1}^{(q^{1/2})}}&
   \dfrac{F_{n-1}(z)}{q^{1/2}\gamma_{n-1}^{(q^{1/2})}}
   \end{bmatrix}, &\text{for}\,z\in\mathcal{D}_{+}, \\
   &\\
 \begin{bmatrix}
   P_n^{(1)}(z) &
   q^{-1/2}(F_n(z)-P_n^{(1)}(z)w(z)h_q(q^{1/2}z)) \\
    \dfrac{P_{n-1}^{(1)}(z)}{\gamma_{n-1}^{(q^{1/2})}}&
   \dfrac{F_{n-1}(z) - P_{n-1}^{(1)}(z)w(z)h_q(q^{1/2}(z)}{q^{1/2}\gamma_{n-1}^{(q^{1/2})}}
   \end{bmatrix},\qquad &\text{for}\, z\in\mathcal{D}_{-}.
  \end{array}\right.
\end{align}
Thus, we have shown that $P_n^{(1)}(z) = P_n^{(q^{1/2})}(z)$. One can readily deduce from Section \ref{statement of rhp section} that
\[\lim_{z\to\infty} \widehat{Y}^{(1)}_{(1,2)}(z) = \frac{\gamma_n^{(1)}}{z^{n+1}} ,\]
it follows from Remark \ref{general c remark},
\[\gamma_{n}^{(q^{1/2})}q^{1/2} = \frac{h_q(q^{1/2}e^{i\pi/4})}{h_q(e^{i\pi/4})}\gamma_n^{(1)}.\]
Finally, we prove that in general $P_n^{(c)}(z)\neq P_n^{(1)}(z)$, if $c\neq q^{1/2},1$. Assume to the contrary that $\{P_n^{(c)}(z)\}_{n=0}^\infty = \{P_n^{(1)}(z)\}_{n=0}^\infty$.  Let, $Y_n^{(c)}$ be the solution of the corresponding RHP given in Remark \ref{general c remark}. Note that the first column of $Y_n^{(1)}$ is the same as that in $Y_n^{(c)}$. 

By Remark \ref{Lax remark}, we know that the second column of $\widehat{Y}_n^{(c)}$ must satisfy the same $q$-difference equation as the second column of $\widehat{Y}_n^{(1)}$, where $\widehat{Y}_n = Y_n$ restricted to $z\in\mathcal{D}_-$. If $\{P_n^{(c)}(z)\}_{n=0}^\infty = \{P_n^{(1)}(z)\}_{n=0}^\infty$ then by the analyticity of $\widehat{Y}_n$ and comparing even and odd terms we conclude the second column of $\widehat{Y}_n^{(c)}$ must satisfy
\begin{gather}\nonumber
\widehat{Y}_n^{(c)}(z)
=
\widehat{Y}_n^{(1)}(z)\begin{bmatrix}
   1 &
   0 \\
   0 &
   C_0
   \end{bmatrix}, \qquad \mathrm{for}\, z\in \mathcal{D}_-,
\end{gather}
where $c_0$ is a constant. Denoting the $(1,2)$-entry of $\widehat{Y}_n^{(c)}(z)$, for $z\in\mathcal{D}_-$, by $\widehat{Y}^{(c)}_{(1,2)}(z)$, we conclude that $\widehat{Y}^{(c)}_{(1,2)}(z) = c_0\widehat{Y}^{(1)}_{(1,2)}(z)$. 

By the analyticity of $Y_n^{(c)}(z)$, we deduce that at the poles of $w(z)$, in particular at $z = e^{i\pi/4},e^{3i\pi/4},e^{-i\pi/4},e^{-3i\pi/4}$, we have
\begin{equation}\label{Yc analytic}
    c_0\widehat{Y}^{(1)}_{(1,2)}(z)+w(z)P_n^{(c)}(z)h(z/c)=0.
\end{equation}
Furthermore, 
\begin{equation}\label{Y1 analytic}
    \widehat{Y}^{(1)}_{(1,2)}(z)+w(z)P_n^{(1)}(z)h(z)=0.
\end{equation}
Equations \eqref{Yc analytic} and \eqref{Y1 analytic} can only be satisfied if the ratio $h(z/c)/h(z)$ is equal at the four points $z = e^{i\pi/4},e^{3i\pi/4},e^{-i\pi/4},e^{-3i\pi/4}$. By Lemma \ref{zero real lemma} and Remark \ref{sym imag}, it follows that the equality can only hold at these four points if $c=1,q^{1/2}$, which is the desired result.
\end{proof}

\begin{theorem}\label{same diff}
Suppose that the sequence of monic polynomials $\{ P^{(c)}_n(x) \}_{n=0}^\infty$ satisfies the orthogonality condition
\[ \int^{\infty}_{-\infty} P^{(c)}_n(cx)P^{(c)}_m(cx)w(cx)d_qx = \gamma^{(c)}_n \delta_{n,m}, \]
where $w(x)$ is given by Equation \eqref{w def} and $q<c\leq1$. Then, the recurrence coefficients $\{ \alpha^{(c)}_{n} \}_{n=1}^\infty$, which occur in the recurrence relation
\begin{equation} \label{ortho2}
    xP^{(c)}_{n}(x) = P^{(c)}_{n+1}(x) + \alpha^{(c)}_{n}P^{(c)}_{n-1}(x),
\end{equation}
solve the equation:
\begin{equation} \label{4structure}
    \alpha_{n}(\alpha_{n+1}+q^{n-1}\alpha_{n}+q^{-2}\alpha_{n-1}-q^{2n-3}\alpha_{n+1}\alpha_{n}\alpha_{n-1}) = (q^{-n}-1)q^{1-n},
\end{equation}
with initial conditions $\alpha_n = 0$ for $n\leq0$. Furthermore, 
\begin{equation}\label{the c deriv}
    D_{q^{-1}}P^{(c)}_{n} = [n]_{q^{-1}}P^{(c)}_{n-1} + \frac{q^{n-3}}{q^{-1}-1}\alpha^{(c)}_{n}\alpha^{(c)}_{n-1}\alpha^{(c)}_{n-2}P^{(c)}_{n-3}.
\end{equation}

\end{theorem}
\begin{proof}
In order to show Equation \eqref{the c deriv}, we observe that
\begin{multline}\label{Boelen 3.9}
    \sum^{\infty}_{k=-\infty} D_{q^{-1}}(P^{(c)}_n(cq^k))P^{(c)}_m(cq^k)w(cq^k)q^k = \\  \sum^{\infty}_{k=-\infty} P^{(c)}_n(cq^k)\frac{(P^{(c)}_m(cq^{k+1})-P^{(c)}_m(cq^k))}{cq^k(q^{-1}-1)}w(cq^k)q^k \\ +\sum^{\infty}_{k=-\infty} \frac{c^4q^{4k}P^{(c)}_n(cq^k)P^{(c)}_m(cq^{k+1})w(cq^k)}{c(q^{-1}-1)}, 
\end{multline}
where we have shifted the index of summation and used Equations \eqref{q_derivative} and \eqref{w diff} to obtain the result.
Applying the orthogonality condition in Equation \eqref{Boelen 3.9}, we find
\begin{equation}
    D_{q^{-1}}P^{(c)}_{n} = A^{(c)}_{n}P^{(c)}_{n-1} + B^{(c)}_{n}P^{(c)}_{n-3},
\end{equation}
where $A^{(c)}_{n}$, $B^{(c)}_{n}$ are constants depending on $n$. We can immediately find $A^{(c)}_{n}$ explicitly by using the identity
\begin{equation}\nonumber
    D_{q^{-1}}(x^{n}) = [n]_{q^{-1}}x^{n-1}.
\end{equation}
Since the sequence $\{ P^{(c)}_{n}(x) \}_{n=0}^\infty$ consists of monic polynomials, we obtain
\[ A^{(c)}_{n}= [n]_{q^{-1}}.\]
In order to derive $B^{(c)}_{n}$, we first note that: 
\begin{equation}\label{the gathering}
    D_{q^{-1}}(xP^{(c)}_{n}) = xq^{-1}D_{q^{-1}}P^{(c)}_{n}(x)+P^{(c)}_{n}(x). 
\end{equation}
We now take the  $q^{-1}$-derivative of both sides of Equation \eqref{ortho2}. After gathering linearly independent terms in Equation \eqref{the gathering} we find
\begin{multline}\nonumber
  (q^{-1}[n]_{q^{-1}}+1-[n+1]_{q^{-1}})P^{(c)}_{n}(x) \\
  + (q^{-1}[n]_{q^{-1}}\alpha^{(c)}_{n-1}+q^{-1}B^{(c)}_{n} - B^{(c)}_{n+1}-\alpha^{(c)}_{n}[n-1]_{q^{-1}})P^{(c)}_{n-2}(x) \\
  + (q^{-1}B^{(c)}_{n}\alpha^{(c)}_{n-3} - \alpha^{(c)}_{n}B^{(c)}_{n-1})P^{(c)}_{n-4}(x) = 0  ,
\end{multline}
which leads to two equations for $B_{i}$ and $\alpha^{(c)}_{j}$:
\begin{eqnarray}
q^{-1}[n]_{q^{-1}}\alpha^{(c)}_{n-1}+q^{-1}B^{(c)}_{n} &=& B^{(c)}_{n+1}+\alpha^{(c)}_{n}[n-1]_{q^{-1}}, \label{lin1}\\
q^{-1}B^{(c)}_{n}\alpha^{(c)}_{n-3} &=& \alpha^{(c)}_{n}B^{(c)}_{n-1}. \label{lin2}
\end{eqnarray}
Equation \eqref{lin2} implies $B^{(c)}_{n} = \hat{c}q^{n}\alpha^{(c)}_{n}\alpha^{(c)}_{n-1}\alpha^{(c)}_{n-2}$, for some constant $\hat{c}$. 

We now proceed to show that $\alpha^{(c)}_n$ satisfies Equation \eqref{4structure}. Substituting $B_n^{(c)}$ into Equation \eqref{lin1} we find
\begin{equation}
    \frac{q^{-n-2}[n]_{q^{-1}}}{\alpha^{(c)}_{n}}-\frac{q^{-n-1}[n-1]_{q^{-1}}}{\alpha^{(c)}_{n-1}} = \hat{c}\bigl(\alpha^{(c)}_{n+1} -q^{-2}\alpha^{(c)}_{n-2}\bigr) .\label{x45term}
\end{equation}
We will rearrange Equation \eqref{x45term} with the goal of obtaining a telescoping sum. Multiplying Equation \eqref{x45term} by $1+d\alpha^{(c)}_{n}\alpha^{(c)}_{n-1}q^{2n-3}$, for some constant $d$, we have
\begin{multline}\nonumber
    \frac{q^{-n-2}[n]_{q^{-1}}}{\alpha^{(c)}_{n}}-\frac{q^{-n-1}[n-1]_{q^{-1}}}{\alpha^{(c)}_{n-1}} +d\alpha^{(c)}_{n-1}q^{n-5}[n]_{q^{-1}} - d\alpha^{(c)}_{n}q^{n-4}[n-1]_{q^{-1}} \\
    = \hat{c}(\alpha^{(c)}_{n+1} -q^{-2}\alpha^{(c)}_{n-2} 
    + dq^{2n-3}\alpha^{(c)}_{n+1}\alpha^{(c)}_{n}\alpha^{(c)}_{n-1} - dq^{2n-5}\alpha^{(c)}_{n}\alpha^{(c)}_{n-1}\alpha^{(c)}_{n-2}) .
\end{multline}
Therefore, we find
\begin{eqnarray*}
    \frac{q^{-n-2}[n]_{q^{-1}}}{\alpha^{(c)}_{n}}-\frac{q^{-n-1}[n-1]_{q^{-1}}}{\alpha^{(c)}_{n-1}} 
    &=& \hat{c}\alpha^{(c)}_{n+1} + d\frac{q^{n-4}}{1-q^{-1}}\alpha^{(c)}_{n} + d\frac{q^{-5}}{1-q^{-1}}\alpha^{(c)}_{n-1}\\
    &&+ \hat{c}dq^{2n-3}\alpha^{(c)}_{n+1}\alpha^{(c)}_{n}\alpha^{(c)}_{n-1} \\
    &&-d\frac{q^{-3}}{1-q^{-1}}\alpha^{(c)}_{n}-d\frac{q^{n-5}}{1-q^{-1}}\alpha^{(c)}_{n-1} - \hat{c}q^{-2}\alpha^{(c)}_{n-2} \\
    &&- \hat{c}dq^{2n-5}\alpha^{(c)}_{n}\alpha^{(c)}_{n-1}\alpha^{(c)}_{n-2}.
\end{eqnarray*}
Letting $\hat{c} = \frac{q^{-3}\tilde{c}}{1-q^{-1}}$, we are led to
\begin{multline*}
    \frac{q^{-n-2}[n]_{q^{-1}}}{\alpha^{(c)}_{n}}-\frac{q^{-n-1}[n-1]_{q^{-1}}}{\alpha^{(c)}_{n-1}} \\
    = \frac{1}{1-q^{-1}}(\tilde{c}q^{-3}\alpha^{(c)}_{n+1} + dq^{n-4}\alpha^{(c)}_{n} + dq^{-5}\alpha^{(c)}_{n-1} + \tilde{c}dq^{2n-6}\alpha^{(c)}_{n+1}\alpha^{(c)}_{n}\alpha^{(c)}_{n-1}) \\    
    -\frac{1}{1-q^{-1}}(dq^{-3}\alpha^{(c)}_{n} + dq^{n-5}\alpha^{(c)}_{n-1} + \tilde{c}q^{-5}\alpha^{(c)}_{n-2} + \tilde{c}dq^{2n-8}\alpha^{(c)}_{n}\alpha^{(c)}_{n-1}\alpha^{(c)}_{n-2}).
\end{multline*}
Thus, if we choose $d=\tilde{c}$ and take the sum from $2$ to $n$ we find
\begin{multline*}
\frac{q^{1-n}(1-q^{-n})}{\alpha^{(c)}_n} - \frac{1-q^{-1}}{\alpha^{(c)}_1} \\
= \tilde{c}(\alpha^{(c)}_{n+1} + q^{n-1}\alpha^{(c)}_{n} + q^{-2}\alpha^{(c)}_{n-1} + \tilde{c}q^{2n-3}\alpha^{(c)}_{n+1}\alpha^{(c)}_{n}\alpha^{(c)}_{n-1} - (\alpha^{(c)}_{2}+\alpha^{(c)}_{1})). 
\end{multline*}
It remains to determine $\tilde{c}$ and $(1-q^{-1})/\alpha^{(c)}_1 -\tilde{c}(\alpha^{(c)}_{2}+\alpha^{(c)}_{1})$. Observe that
\begin{eqnarray}\label{eqmu4}
    \sum_{k=-\infty}^\infty D_{q^{-1}}((cq^k)^n)w(cq^k)q^k &=& \sum_{k=-\infty}^\infty \frac{(cq^k)^{n+3}w(cq^k)q^k}{q^{-1}-1},
\end{eqnarray}
where we have used Equations \eqref{q_derivative} and \eqref{w diff}. 

We now use Equation \eqref{eqmu4} to determine $\tilde{c}$ and $(1-q^{-1})/\alpha^{(c)}_1 -\tilde{c}(\alpha^{(c)}_{2}+\alpha^{(c)}_{1})$. Without loss of generality assume that $\int_{-\infty}^\infty w(cx)d_qx=1$.  From Equation \eqref{eqmu4} we find that
\begin{eqnarray}
    \int_{-\infty}^\infty (cx)^4w(cx)d_qx &=&  (q^{-1}-1). \label{8.9} \\
    \int_{-\infty}^\infty (cx)^6w(cx)d_qx &=&  (q^{-3}-1)\int_{-\infty}^\infty (cx)^2w(cx)d_qx .\label{8.10}
\end{eqnarray}
Using Equation \eqref{ortho2} we determine
\begin{eqnarray}
    \alpha^{(c)}_1 &=& \int_{-\infty}^\infty (cx)^2w(cx)d_qx. \label{8.11} \\
    \alpha^{(c)}_3 &=& \frac{\int_{-\infty}^\infty P^{(c)}_3(cx)^2w(cx)d_qx}{\int_{-\infty}^\infty P^{(c)}_2(cx)^2w(cx)d_qx},\label{8.12} \\
    &=& \frac{\int_{-\infty}^\infty \bigl((cx)^6 -2(\alpha^{(c)}_1+\alpha^{(c)}_2)(cx)^4 + (\alpha^{(c)}_1+\alpha^{(c)}_2)^2(cx)^2)w(cx)d_qx }{\int_{-\infty}^\infty ((cx)^4-2\alpha^{(c)}_1(cx)^2+(\alpha^{(c)}_1)^2\bigr)w(cx)d_qx}.\nonumber
\end{eqnarray}
Furthermore, by the orthogonality of $\{P^{(c)}_n(x)\}_{n=0}^\infty$ we find
\begin{eqnarray}
    \int_{-\infty}^\infty P^{(c)}_1(cx)P^{(c)}_3(cx)w(cx)d_qx &=&  \int_{-\infty}^\infty \bigl((cx)^4-(\alpha^{(c)}_1+\alpha^{(c)}_2)(cx)^2\bigr)w(cx)d_qx,\nonumber  \\
     &=& 0 \label{8.13},
\end{eqnarray}
where we have used the recurrence relation, Equation \eqref{ortho2}, to determine $P^{(c)}_3(x)$. Applying Equations \eqref{8.9} to \eqref{8.13} we deduce
\begin{eqnarray*}
\tilde{c} &=& -1. \\
q^{-1}-1 &=& \alpha^{(c)}_1(\alpha^{(c)}_1+\alpha^{(c)}_2) .
\end{eqnarray*}
These values give Equations \eqref{4structure} and \eqref{the c deriv} and Theorem \ref{same diff} follows immediately.
\end{proof}
As a consequence of Theorem \ref{same diff} the sequences $\{\alpha^{(c)}_n\}_{n=1}^\infty$  all provide positive solutions of Equation \eqref{4structure}. However, we will show that the limits of these sequences, as $n\to\infty$, are in general not the same. In particular, we show that the asymptotic limit as $n\to\infty$ of $\alpha^{(1)}_n$ does not equal the limit of $\alpha^{(c)}_n$ if $c\neq1,q^{1/2}$. 

\begin{theorem}\label{theorem diff asym}
Suppose the sequence of orthogonal polynomials $\{ P_n^{(c)}(z) \}_{n=0}^\infty$ and recurrence coefficients $\{\alpha^{(c)}_n\}_{n=1}^\infty$ are defined as in Theorem \ref{same diff}. Then, it follows that
\begin{equation}\label{equ 8.17}
    \lim_{n\to\infty} \frac{\alpha^{(c)}_n - \alpha^{(1)}_n}{q^{1-n}- \alpha^{(1)}_n} \neq 0,
\end{equation}
if $c\neq1,q^{1/2}$.
\end{theorem}
\begin{proof}
From Lemma \ref{lemma 8.1}, we have $\{ \alpha^{(1)}_n \}_{n=0}^\infty = \{ \alpha^{(q^{1/2})}_n \}_{n=0}^\infty$. It remains to study the case $c\neq1,q^{1/2}$.
Let $y_1(z)$ and $y_2(z)$ be two solutions of Equation \eqref{near 0 q diff}. Then, $Y(z) = [y_1(z),y_2(z)]^T$ is a solution of the matrix equation
\begin{gather}\label{Y near 0 diff}
Y(q^{-1}z)
=
\begin{bmatrix}
1-q^{-2}z^2 &
q^{-2} z \\
-q^{-2}z &
1-q^{-2}z^2 
\end{bmatrix}
Y(z).
\end{gather}
Using the definition of the $q$-derivative, and Equation \eqref{ortho2} to write $P_{n-3}(x)$ in terms of $P_n(x)$ and $P_{n-1}(x)$, Equation \eqref{the c deriv} can be re-written as
\begin{multline}\label{Pn q diff}
      P^{(c)}_n(q^{-1}z) = (1-q^{n-3}z^2\alpha^{(c)}_n)P_n^{(c)}(z) \\ + ((q^{-n}-1)z-z\alpha^{(c)}_n\alpha^{(c)}_{n-1}q^{n-3} + z^3 \alpha^{(c)}_{n}q^{n-3}) P^{(c)}_{ n-1}(z).
\end{multline}
Taking $n\mapsto n-1$ and again using Equation \eqref{ortho2}, Equation \eqref{Pn q diff} also allows us to express $P^{(c)}_{n-1}(q^{-1}z)$ in terms of $P^{(c)}_n(z)$ and $P^{(c)}_{n-1}(z)$. This results in the matrix difference equation
\begin{align*}\label{limit P diff}
\begin{bmatrix}
P^{(c)}_n(q^{-1}z)\\
q^{-n/2}P^{(c)}_{n-1}(q^{-1}z)
\end{bmatrix}
=&
\left(\begin{bmatrix}
1-q^{n-3}\alpha^{(c)}_nz^2 &
l_n z \\
k_n z &
1-q^{n-4}\alpha^{(c)}_{n-1}z^2 
\end{bmatrix}+r(z)\right)\cdot\\
&\qquad\qquad \cdot\begin{bmatrix}
P^{(c)}_n(z)\\
q^{-n/2}P^{(c)}_{n-1}(z)
\end{bmatrix},
\end{align*}
where
\begin{eqnarray*}
l_n &=& q^{-n/2}-\alpha^{(c)}_n \alpha^{(c)}_{n-1}q^{\frac{3n}{2}-3}, \\
k_n &=& q^{1-\frac{3n}{2}}/\alpha^{(c)}_{n-1} - \alpha^{(c)}_{n-2}q^{\frac{n}{2}-4}, \\
\end{eqnarray*}
and $r(z)$ is given by
\begin{gather}
r(z) = z\begin{bmatrix}
0 &
z^2a_nq^{\frac{3n}{2}-3}-q^{n/2} \\
q^{\frac{n}{2}-4}z^2 -q^{-n/2}(\alpha^{(c)}_{n-1})^{-1}&
zq^{1-n}/\alpha^{(c)}_{n-1} - z\alpha^{(c)}_{n-2}q^{n-4} + z^3q^{n-4}
\end{bmatrix}.
\end{gather}
For the case $c=1$ we know that as $n\to\infty$ Equation \eqref{limit P diff} approaches Equation \eqref{Y near 0 diff} (this follows from the the proof that $W_n(z)\sim G(z)$ as $n\to\infty$ , see Section \ref{main proof section}). Assume that $\alpha^{(c)}_n$ has the same asymptotic behaviour as $\alpha^{(1)}_n$. Then Equation \eqref{limit P diff} must similarly approach Equation \eqref{Y near 0 diff}. In particular, $r(z)=o(1)$ as $n\to\infty$ and, $l_n$, $k_n$ approach non-zero constants. Note that the condition $r(z)=o(1)$ as $n\to\infty$ follows from the first order asymptotic behaviour i.e. $\alpha^{(c)}_n \sim q^{1-n}$ as $n\to\infty$. In order for $l_n$ and $k_n$ to approach the same non-zero constant as the case $c=1$ we require the asymptotic behaviour to match to second order.

We conclude using Remarks \ref{Lax remark} and \ref{general c remark}, and similar arguments to those of Lemma \ref{a/g lemma}, that for $z\in \mathcal{D}_+$, $Y^{(c)}_{2n}(z)$ approaches the matrix
\begin{gather}
Y^{(c)}_{2n}(z)
\sim
\begin{bmatrix}
a(z) &
w(z)(h_q(z/c)a(z)-\lambda_1b(z)) \\
b(z) &
w(z)(h_q(z/c)b(z)-\lambda_2a(z)) 
\end{bmatrix},\;\text{as}\,n\to\infty,
\end{gather}
for some real constants $\lambda_{1,2}$ (recall $a(z)$ and $b(z)$ are defined in Lemma \ref{L01}). By the meromorphicity of $Y^{(c)}_{2n}(z)$ at the poles of $w(z)$ we find
\begin{equation}
    h_q(z/c)a(z)-\lambda_1b(z)=h_q(z/c)b(z)-\lambda_2a(z)=0.
\end{equation}
Thus,
\begin{equation}\label{h imag}
    h_q(z/c)^2=\lambda_1\lambda_2.
\end{equation}
Hence, $h_q(z/c)$ is either real or imaginary. We conclude from Lemma \ref{ray zeros} and Remark \ref{sym imag} that Equation \eqref{h imag} is satisfied iff $c=1,q^{1/2}$. It follows that $\alpha^{(c)}_n$ has the same asymptotic behaviour as $\alpha^{(1)}_n$ (i.e. Equation \eqref{equ 8.17} is zero) iff $c=1,q^{1/2}$.
\end{proof}

\section{Conclusion}
In this paper, we described the asymptotic behaviour of a class of $qF_{II}$ polynomials, defined in Section \ref{Non-unique measure section}, by using the $q$-RHP setting \cite{qRHP}. Our main results are Theorems \ref{main result 1}, \ref{main result 2} and \ref{main result 3}. In Theorems \ref{main result 1} and \ref{main result 2}, we provided detailed asymptotic results for $qF_{II}$ polynomials. In Theorem \ref{main result 3}, we detailed the implications of our analysis for the $q$-Painlev\'e equation satisfied by the recurrence coefficients, $\{ \alpha_n \}_{n=1}^\infty$, of $qF_{II}$ polynomials (see Equation \eqref{general recurrence coefficients}).

Perhaps the most unexpected results of this paper concern the effect of the properties of $h_q(z)$  on the class of orthogonal polynomials that arise when the lattice was varied. The values of $h_q(z)$ at the poles of the weight function $w(z)$ play an important role in determining the behaviour of orthogonal polynomials supported on the shifted lattice $cq^k$, $k\in\mathbb{Z}$. This observation enabled us to determine whether the class of orthogonal polynomials were invariant as $c$ varies. Furthermore, we were able to compare the variation in the asymptotic behaviours of polynomials when the lattice was shifted. 

This paper focused on the weight $(-x^{4};q^{4})_\infty^{-1}$. But, the methodology can readily be extended to describe discrete $q$-Hermite II polynomials \cite[Chapter 18.27]{NIST:DLMF} with weight $(-x^{2};q^{2})_\infty^{-1}$. However, generalising the results to higher order weights i.e. $(-x^{2m};q^{2m})_\infty^{-1}$, for $m>2$, remains an open problem. The key difficulty is describing and solving the near-field RHP for higher order weights. One important aspect of this problem is accounting for the increased number of poles when dealing with higher order weights.

Another possible direction of future research could be determining all of the positive solutions of Equation \eqref{qPainleve 4} and the asymptotic behaviour of different solutions as $n\to\infty$. In this paper we described the asymptotic behaviour of one particular solution, which satisfies the limit
\begin{equation}\label{conclusion an limit}
    \lim_{n\to\infty}q^n \alpha^{(1)}_n = q.
\end{equation} 
Numerical evidence suggests that the RHS of Equation \eqref{conclusion an limit} oscillates for other positive solutions of Equation \eqref{qPainleve 4}. It would be interesting to see if the $q$-RHP formalism is able to accurately capture this behaviour. One avenue to achieve this would be to extend the detailed asymptotic results obtained in this paper to $qF^{(c)}_{II}$ polynomials, defined in Section \ref{Non-unique measure section}.

\appendix

\section{Properties of $h_q(z)$}\label{Properties of hq}
In this section we prove some properties of the function $h_q(z)$ defined in Equation \eqref{h new def}, which are used in this paper. Before discussing $h_q(z)$ we first prove a necessary Lemma.

\begin{lemma}\label{lemma 5.6}
Let $C(z)$ be a function defined on $\mathbb{C}\setminus \{0\}$, which is analytic everywhere except for simple poles at $q^k$ for $k \in \mathbb{Z}$. Then, $C(qz) \neq C(z)$. 
\end{lemma}

\begin{proof}
We prove the result by contradiction. Assume $C(qz) = C(z)$. Define
\[ G(z) = (-z,-qz^{-1};q)_\infty .\]
By direct calculation one can show $G(qz) = z^{-1}G(z)$. Furthermore, by definition, $G(z)$ is zero on the $q$-lattice $q^k$, $k \in \mathbb{Z}$. Let
\[ F(z) = C(z)G(z),\]
then it follows $F(z)$ is analytic in $\mathbb{C}\setminus \{0\}$ and satisfies the difference equation
\begin{equation}\label{cF diff eq}
    F(qz) = z^{-1}F(z) .
\end{equation}
As $F(z)$ is analytic in $\mathbb{C}\setminus \{0\}$ we can write $F(z)$ as the Laurent series
\[ F(z) = \sum_{k=-\infty}^\infty F_kz^k.\]
Comparing the coefficients of $z$ in Equation \eqref{cF diff eq}, one can readily determine
\begin{equation}\label{fk eq}
    F_k = c_0q^{k(k-1)/2}.
\end{equation} 
However, there is only one solution with $z$ coefficients given by Equation \eqref{fk eq} (up to scaling by a constant) and it follows that $F(z) = c_0G(z)$. Thus, if $C(qz) = C(z)$, then $C(z) = c_0$, and $C(z)$ has no poles. 
\end{proof}
\begin{corollary}\label{even coror}
Let $C(z)$ be a function defined on $\mathbb{C}\setminus \{0\}$, which is analytic everywhere except for simple poles at $\pm q^k$ for $k \in \mathbb{Z}$. Furthermore, suppose $C(z)$ satisfies $C(qz)=C(z)$. Then, $C(z) = c_1h_q(z) + c_0$, where $c_0$ and $c_1$ are constants and $h_q(z)$ is as defined in Definition \ref{formal h def}.
\end{corollary}
\begin{proof}
As both $C(z)$ and $h_q(z)$ have simple poles at $z=-1$ we conclude that there exists a $c_1\neq 0$ such that
\[ \mathrm{Res}(C(-1)) = c_1\mathrm{Res}(h_q(-1)).\]
Furthermore, both $C(z)$ and $h_q(z)$ are invariant under the transformation $z\to qz$, hence for all $k\in\mathbb{Z}$
\[ \mathrm{Res}(C(-q^k)) = c_1\mathrm{Res}(h_q(-q^k)).\]
Thus, the function 
\[D(z) = C(z)-c_1h_q(z),\]
is meromorphic in $\mathbb{C}\setminus \{0\}$, with possible simple poles at $q^k$ for $k \in \mathbb{Z}$, and satisfies $D(qz) = D(z)$. However, by Lemma \ref{lemma 5.6}, $D(z)$ can not have simple poles at $q^k$ for $k \in \mathbb{Z}$. Hence, $D(z)$ is analytic in $\mathbb{C}\setminus \{0\}$ and it follows that $D(z)$ can be written as a convergent Laurent series. Thus,
\[ D(z) = \sum_{j=-\infty}^\infty d_jz^{j} .\]
Substituting this into the $q$-difference equation $D(qz) = D(z)$, we conclude $D(z) = d_0$ $(=c_0)$ and Corollary \ref{even coror} follows immediately.
\end{proof}

\begin{lemma}\label{zero real lemma}
The function $h_q(z)$ has zero real part along the circles $|z| = 1$ and $|z| = q^{1/2}$. 
\end{lemma}

\begin{proof}
Let $z = re^{i\theta}$. Substituting this into Equation \eqref{h new def} and determining the real part we find
\begin{equation}\label{real h}
\mathrm{Re}(h_q(re^{i\theta})) = \sum_{k=-\infty}^\infty \frac{2rq^k cos(\theta)(r^2-q^{2k})}{r^4+q^{4k} - 2r^2q^{2k}cos(2\theta)}. 
\end{equation}
First we consider the case $r=1$. The RHS of Equation \eqref{real h} becomes
\begin{eqnarray*}
\sum_{k=-\infty}^\infty \frac{2q^k cos(\theta)(1-q^{2k})}{1+q^{4k} - 2q^{2k}cos(2\theta)} &=&  \sum_{k=1}^\infty \frac{2q^k cos(\theta)(1-q^{2k})}{1+q^{4k}- 2q^{2k}cos(2\theta)}  \\
&& \, + \sum_{k=-\infty}^{-1} \frac{2q^k cos(\theta)(1-q^{2k})}{1+q^{4k} - 2q^{2k}cos(2\theta)}, \\
&=&  \sum_{k=1}^\infty \frac{2q^k cos(\theta)(1-q^{2k})}{1+q^{4k}- 2q^{2k}cos(2\theta)}  \\
&& \, + \sum_{k=1}^{\infty} \frac{2q^{-k} cos(\theta)(1-q^{-2k})}{1+q^{-4k} - 2q^{-2k}cos(2\theta)},\\
&=&0.
\end{eqnarray*}
Next we consider the case $r = q^{1/2}$. The RHS of Equation \eqref{real h} can be written as

\begin{eqnarray*}
\sum_{k=-\infty}^\infty \frac{2q^{k+1/2} cos(\theta)(q-q^{2k})}{q^2+q^{4k} - 2q^{2k+1}cos(2\theta)} &=&  \sum_{k=1}^\infty \frac{2q^{k+1/2} cos(\theta)(q-q^{2k})}{q^2+q^{4k} - 2q^{2k+1}cos(2\theta)}  \\
&& \, + \sum_{k=-\infty}^{0} \frac{2q^{k+1/2} cos(\theta)(q-q^{2k})}{q^2+q^{4k} - 2q^{2k+1}cos(2\theta)}, \\
&=&  \sum_{k=1}^\infty \frac{2q^{k+1/2} cos(\theta)(q-q^{2k})}{q^2+q^{4k} - 2q^{2k+1}cos(2\theta)}  \\
&& \, + \sum_{k=0}^{\infty} \frac{2q^{-k+1/2} cos(\theta)(q-q^{-2k})}{q^2+q^{-4k} - 2q^{-2k+1}cos(2\theta)}, \\
&=&  \sum_{k=1}^\infty \frac{2q^{k+1/2} cos(\theta)(q-q^{2k})}{q^2+q^{4k} - 2q^{2k+1}cos(2\theta)}  \\
&& \, + \sum_{j=1}^{\infty} \frac{2q^{j+1/2} cos(\theta)(q^{2j}-q)}{q^2+q^{4j} - 2q^{2j+1}cos(2\theta)}, \\
&=&0.
\end{eqnarray*}
\end{proof}

\begin{corollary}\label{corollary a.4}  
The function $h_q(z)$ defined in Equation \eqref{h new def} can also be written as
\begin{equation}\label{diffferent hq rep}
    h_q(z) = c_1\frac{z(qz^2,qz^{-2};q^2)_\infty}{(z^2,q^2z^{-2};q^2)_\infty},
\end{equation}
for some constant $c_1$.
\end{corollary}
\begin{proof}
We observe that the function
\begin{equation}\nonumber
    h(z) = \frac{z(qz^2,qz^{-2};q^2)_\infty}{(z^2,q^2z^{-2};q^2)_\infty},
\end{equation}
is meromorphic with simple poles at $\pm q^k$ for $k \in \mathbb{Z}$, and satisfies $h(qz)=h(z)$. By Corollary \ref{even coror} we conclude $h(z) = c_1h_q(z) + c_0$. By definition, $h(z)$ has zeros at $z = q^{1/2+k}$, for $k \in \mathbb{Z}$, and by Lemma \ref{zero real lemma}, $h_q(z)$ also has zeros at $z = q^{1/2+k}$, for $k \in \mathbb{Z}$. Thus, we conclude 
\begin{equation}\nonumber
    h_q(z) = c_1\frac{z(qz^2,qz^{-2};q^2)_\infty}{(z^2,q^2z^{-2};q^2)_\infty},
\end{equation}
for some constant $c_1$.
\end{proof}

\begin{lemma}\label{ray zeros}
Along the ray $z = re^{i\pi/4}$, $(r\in\mathbb{R}_{\geq 0})$ the real part of $h_q(z)$ is non-zero except at $r = q^{k/2}$, for $k\in \mathbb{Z}$ where $h_q(z)$ is complete imaginary.
\end{lemma}
\begin{proof}
From Equation \eqref{real h} we determine that the real part of $h_q(re^{i\pi/4})$ is given by
\[ \mathrm{Re}(h_q(re^{i\pi/4})) = \sum_{k=-\infty}^\infty \frac{\sqrt{2}rq^k(r^2-q^{2k})}{r^4+q^{4k}}. \]
Define the function
\[ F(u) = \sum_{k=-\infty}^\infty \frac{uq^k(u^2-q^{2k})}{u^4+q^{4k}}, \]
where $u$ is a complex variable. We note that the sum is well defined and converges for all $u \neq q^{k}e^{(i\pi + 2n\pi)/4}$, where $k\in\mathbb{Z}$ and $n = 0,1,2,3$. Furthermore, $F(u)$ satisfies the $q$-difference equation $F(qu) = F(u)$. We now multiply $F(u)$ by a function with zeros at the poles of $F(u)$ to give us a function analytic in $\mathbb{C}\setminus \{0\}$. Define
\[ g(u) = F(u)(-u^4,q^4u^{-4};q^4)_\infty.\]
Note that $(-u^4,q^4u^{-4};q^4)_\infty$ is an even function and satisfies $(-(qu)^4,q^4(qu)^{-4};q^4)_\infty = u^{-4}(-u^4,q^4u^{-4};q^4)_\infty$. Thus, $g(u)$ is analytic in $\mathbb{C}\setminus \{0\}$, $g(u)$ satisfies the $q$-difference equation $g(qu) = u^{-4}g(u)$, and $g(u)$ is an odd function. Let us represent $g(u)$ with the convergent Laurent series
\[ g(u) = \sum_{j=-\infty}^\infty g_ju^j .\]
As $g(u)$ satisfies the difference equation $g(qu) = u^{-4}g(u)$ we find
\[ g_{j+4} = q^jg_j. \]
Hence, we conclude that there are four linearly independent solutions, which can be chosen such that two are odd and two are even. As $g(u)$ is odd we conclude it is the sum of two linearly independent odd solutions. Consider the two functions
\[ G_1(u) = u(u^2,q^2u^{-2};q^2)_\infty (qu^2,qu^{-2};q^2)_\infty,\]
\[ G_2(u) = u(-u^2,-q^2u^{-2};q^2)_\infty (-qu^2,-qu^{-2};q^2)_\infty.\]
We note that $G_1(u)$ has zeros at $u = \pm q^{k/2}$, for $k \in \mathbb{Z}$ and $G_2(u)$ has zeros at $u = \pm iq^{k/2}$, for $k \in \mathbb{Z}$. Furthermore, both $G_1(u)$ and $G_2(u)$ are odd and satisfy the difference equation $G(qu) = u^{-4}G(u)$. Thus, 
\[ g(u) = c_1G_1(u) + c_2G_2(u), \]
for some constants $c_1$ and $c_2$. By Lemma \ref{zero real lemma} we conclude that $c_2=0$ and $g(u)$ only has zeros at the zeros of $G_1(u)$ which occur at $u = \pm q^{k/2}$, for $k \in \mathbb{Z}$. Hence, this is where the zeros of $F(u)$ are and Lemma \ref{ray zeros} follows immediately.
\end{proof}

\begin{remark}\label{sym imag}
From Equation \eqref{real h} we deduce that
\[ \mathrm{Re}(h_q(re^{i\pi/4})) = -\mathrm{Re}(h_q(re^{3i\pi/4})) = -\mathrm{Re}(h_q(re^{-3i\pi/4})) = \mathrm{Re}(h_q(re^{-i\pi/4})).\]
Using an analogous expression to Equation \eqref{real h}, for the imaginary part of $h_q(z)$, one can also show by direct calculation that
\[ \mathrm{Im}(h_q(re^{i\pi/4})) = \mathrm{Im}(h_q(re^{3i\pi/4})) = -\mathrm{Im}(h_q(re^{-3i\pi/4})) = -\mathrm{Im}(h_q(re^{-i\pi/4})) \neq 0.\]
\end{remark}

\section{RHP theory}\label{R to I section}
For completeness we recall a well known result from RHP theory and prove it below. Let $R(z)$ be a solution of the following RHP:
\begin{definition}\label{example RHP}
Let $\Gamma$ be an appropriate curve (see Definition \ref{admissable}) with interior $\mathcal D_-$ and exterior $\mathcal D_+$. A $2\times 2$ complex matrix function $R(z)$, $z\in\mathbb C$, is a solution of the RHP (\ref{example RHP}) if it satisfies the following conditions:
\begin{enumerate}[label={{\rm (\roman *)}}]
\begin{subequations}
\item $R(z)$ is analytic in $\mathbb{C}\setminus \Gamma$.
\item $R(z)$ has continuous boundary values $R^-(s)$ and $R^+(s)$ as $z$ approaches $s\in\Gamma$ from $\mathcal D_-$ and $\mathcal D_+$ respectively, where 
\begin{gather} \label{1.6a}
R^+(s)
=
R^-(s)J(s), \; s\in\Gamma,
\end{gather}
for a $2\times2$ matrix $J(s)$.

\item $R(z)$ satisfies
\begin{gather} \label{1.6c}
R(z) = 
\begin{bmatrix}
1 &
0  \\
0&
1
\end{bmatrix} + O\left(\frac{1}{z}\right),\; as\, |z| \to \infty.
\end{gather}

\end{subequations} 
\end{enumerate} 
\end{definition}

\begin{theorem}\label{R to 1 theorem}
Suppose that $J(s)$ can be analytically extended to a neighbourhood of $\Gamma$. Furthermore, for a given $0<\epsilon<1$ suppose that 
\begin{equation} \nonumber
||J(s) - I|| < \epsilon,
\end{equation}  
in this neighbourhood (where $||.||$ is the matrix norm). Then, the solution of the RHP given by Definition \ref{example RHP} satisfies
\begin{equation} \nonumber
||R(z) - I|| < O(\epsilon),
\end{equation}
for all $z\in \mathbb{C}$.
\end{theorem}
\begin{proof}
Multiple sources give a proof of Theorem \ref{R to 1 theorem} with various conditions on the jump matrix $J$ \cites{Deift1999strong}{kuijlaars2003riemann}. For our setting the jump matrix is quite well behaved and satisfies all these constraints. We include a brief proof for completeness. Let 
\[ \Delta(s) = R(s) - I ,\]
substituting $\Delta(s)$ into Equation \eqref{1.6a} gives
\begin{equation*}
    R^+(s) = R^-(s)(I+\epsilon\Delta(s)).
\end{equation*}
By the asymptotic condition, Equation \eqref{1.6c}, we conclude that
\begin{equation}\label{R outside}
    R(z) = I +\frac{\epsilon}{2\pi i}\oint_{\Gamma}\frac{R_-(s)\Delta(s)}{z-s}ds. 
\end{equation}
Let $L$ be defined as $L = \sup_{z\in\mathcal{D}_-}(|R(z)|)$, and let the maximum be at $z_L$. As $R(z)$ is analytic in $\mathcal{D}_-$ it follows $|R(z)|$ achieves its maximum on the boundary (i.e. on $\Gamma$). By assumption $R(z)$ and $\Delta(z)$ are also analytic for some fixed distance $r$ from $\Gamma$, let us call this curve $\Gamma_i$. Therefore,
\begin{equation}\label{rxl}
    R(z_L) = \left(I + \frac{\epsilon}{2\pi i}\oint_{\Gamma_i}\frac{R_-(s)\Delta(s)}{z-s}ds\right)\left(I+\epsilon\Delta(s)\right)^{-1},
\end{equation}
where $\left(I+\epsilon\Delta(s)\right)^{-1}$ can be determined using the Nuemann series 
\begin{equation*}
    \left(I+\epsilon\Delta(s)\right)^{-1} = \sum_{j=0}^{\infty}(-\epsilon\Delta(s))^{j}.
\end{equation*}
We conclude from Equation \eqref{rxl}
\begin{equation}
    L < \left| I + \frac{\epsilon \Vert \Delta \Vert_{\Gamma_i} \mathrm{len}(\Gamma_i)}{2\pi r} \right| \left|\sum_{j=0}^{\infty}(-\epsilon \Vert \Delta \Vert_{\Gamma})^{j}\right|,
\end{equation}
and hence $|L-1| = O(\epsilon)$. Thus we find that,
\begin{equation}
    |R(z)-I| < c_\Gamma O(\epsilon)
\end{equation}
For some constant 
\[ c_\Gamma = \frac{\Vert \Delta \Vert_{\Gamma_i} \mathrm{len}(\Gamma_i)}{2\pi r} + \Vert \Delta \Vert_{\Gamma} ,\]
which is independent of $\epsilon$.
\end{proof}

\section*{Funding}
Nalini Joshi's research was supported by an Australian Research Council Discovery Projects \#DP200100210 and \#DP210100129. Tomas Lasic Latimer's research was supported the Australian Government Research Training Program and by the University of Sydney Postgraduate Research Supplementary Scholarship in Integrable Systems.

\printbibliography

\end{document}